\documentclass[reqno,11pt]{amsart}
\usepackage{amscd,amsfonts,mathrsfs,amsthm,enumerate}
\usepackage{amssymb, amsmath}
\usepackage{stmaryrd}
\usepackage{epsfig}
\usepackage[sorted-cites,nobysame]{amsrefs}
\usepackage{latexsym}
\usepackage[usenames,dvipsnames]{color}
\usepackage[colorlinks=true,linkcolor=blue,citecolor=BrickRed,urlcolor=RoyalBlue]{hyperref}
\def\real     #1{{\mathbb R^{#1}}}

\def\equationcolor {\color{black}}
\def\textcolor     {\color{black}}

\def\bcoleq    {\begin{equation}\equationcolor}
\def\ecoleq    {\textcolor\end{equation}}
\def\bcoleqn   {\equationcolor\begin{eqnarray}}
\def\ecoleqn   {\end{eqnarray}\textcolor}

\def\cosec {\operatorname{cosec}}
\def\sec {\operatorname{sec}}

\def\C        {\mathbb{C}}
\def\S        {\mathbb{S}}
\def\CP        {\mathbb{CP}}

\def\gm{{\operatorname{g}_M}}
\def\gn{{\operatorname{g}_N}}

\def\gk{{\operatorname{g}_{M\times N}}}
\def\rn{{{R}_N}}
\def\rk{{\tilde{R}}}

\def\rind{{R}}

\def\gind{\operatorname{g}}

\DeclareMathOperator*{\rank}{rank}

\newtheorem{theorem}{Theorem}[section]
\newtheorem{mythm}{Theorem}

\newtheorem{lemma}[theorem]{Lemma}

\newtheorem*{thma}{Theorem A}
\newtheorem*{thmb}{Theorem B}
\newtheorem*{thmc}{Theorem C}

\newtheorem{corollary}[theorem]{Corollary}
\newtheorem{proposition}[theorem]{Proposition}

\theoremstyle{definition}
\newtheorem*{assumption*}{$\lambda_{1}$-Condition}

\newtheorem{remark}[theorem]{Remark}

\def\pproof#1{\@ifnextchar[\opargproof
{\opargproof[\it Proof of #1.]}}
\def\opargproof[#1]{\par\noindent {\bf #1 }}

\makeatother

\numberwithin{equation}{section}

\begin{document}

\title[Rigidity of the Hopf fibration]{Rigidity of the Hopf fibration}
\author[M. Markellos]{\textsc{Michael Markellos}}
\address{Michael Markellos\newline
University of Ioannina\newline
Section of Algebra \& Geometry\newline
45110 Ioannina, Greece\newline
{\sl E-mail address:} {\bf mmarkellos@hotmail.gr}
}
\author[A. Savas-Halilaj]{\textsc{Andreas Savas-Halilaj}}
\address{Andreas Savas-Halilaj\newline
University of Ioannina\newline
Section of Algebra \& Geometry\newline
45110 Ioannina, Greece\newline
{\sl E-mail address:} {\bf ansavas@uoi.gr}
}

\renewcommand{\subjclassname}{  \textup{2000} Mathematics Subject Classification}
\thanks{The authors would like to acknowledge support by the General Secretariat for Research and Innovation (GSRI) and the
Hellenic Foundation for Research and Innovation (HFRI) Grant No:133.}
\parindent = 0 mm
\hfuzz     = 6 pt
\parskip   = 3 mm
\date{}

\begin{abstract}
In this paper, we study minimal maps between euclidean spheres.
The Hopf fibrations provide explicit examples of such minimal maps. Moreover, their
corresponding graphs have second fundamental form of constant norm. We prove that a minimal submersion
from $\S^3$ to $\S^2$ whose Gauss map satisfies a suitable pinching condition must be weakly conformal and with
totally geodesic fibers. As a
consequence, we obtain that an equivariant minimal submersion from $\S^3$ to $\S^2$ coincides with the
Hopf fibration. Furthermore, we prove that a minimal map $f:\S^3\to\S^2$ with constant singular values and constant norm of the second fundamental form is either
constant or, up to isometries, coincides with the Hopf fibration.
\end{abstract}

\maketitle
\section{Introduction}
Let $f:M\to N$ be a smooth map between two manifolds $M$ and $N$. It is
a fundamental problem to find {\em canonical representatives} in the homotopy class of
$f$. By a canonical representative is usually meant a map in the homotopy class of the given
map $f$ which is a critical point of a suitable functional. In the mid-1960's, Eells \& Sampson \cite{eells} introduced the {\it harmonic maps} as critical points of the energy functional, in order to attack the
aforementioned problem. However, it is not so easy to produce harmonic maps between Riemannian
manifolds, especially between spheres.

Another important functional that we may consider in the space of smooth maps is the volume functional.
Given a smooth map $f:M\to N$ between Riemannian manifolds
$(M,\gm)$ and $(N,\gn)$, let us denote its graph in the Riemannian product space $M\times N$ by
$$\varGamma(f):=\{(x,f(x))\in M\times N:x\in M\}.$$
Following the terminology introduced by Schoen \cite{schoen2}, a map whose graph is minimal submanifold
is called a {\em minimal map}.

There is a very interesting class of non-trivial maps from odd-dimensional spheres to projective
spaces, the so-called {\em Hopf fibrations} \cite{hopf}; see Section \ref{DGHFalex}. These
maps are described by the actions
$$
\mathbb{S}^{1}\hookrightarrow\mathbb{S}^{2n+1}\to\mathbb{CP}^n,
\quad \mathbb{S}^3\hookrightarrow\mathbb{S}^{4n+3}\to\mathbb{QP}^n
\quad\text{and}\quad\mathbb{S}^7\hookrightarrow\mathbb{S}^{15}\to\mathbb{OP}^{1}.
$$
The fibers of the Hopf fibrations are great circles of the euclidean sphere. Submersions from spheres whose fibers are geodesic circles
are called {\it great circle fibrations}. The space of all great circle fibrations
(not necessarily harmonic or minimal) of odd-dimensional spheres
is infinite dimensional. In particular, there is an abundance of such maps that are not
obtained by linear repositioning of the Hopf fibration.
According to beautiful results of Gluck \& Warner
\cite{gluck3}, Yang
\cite{yang1,yang2}, and McKay \cite{mckay}, for any great circle submersion of
$\mathbb{S}^{2n+1}$ into $\mathbb{CP}^n$,
there is a diffeomorphism of the sphere $\mathbb{S}^{2n+1}$  carrying it to the Hopf fibration.

One aim of this paper is to understand the Hopf fibrations in geometric terms, to explore their
geometry, and to show that their graphs are rigid in a suitable sense. It turns out that the Hopf fibrations are minimal
Riemannian submersions and their graphs have second fundamental form of constant norm; see Section \ref{DGHFalex}.
In the case of minimal submersions from $\S^3$ into $\S^2$ we show that the converse is also true.
More precisely, we show the following:

\begin{mythm}\label{thmc}
Let $f:U\to\mathbb{S}^2$ be a minimal map, where $U$ is an open subset of $\S^3$.
If $f$ has constant singular values and $\varGamma(f)$ constant norm of the second fundamental form, then the singular
values are equal. If additionally,  $U=\S^3$ and $f$ is non-constant, then there exists an isometry $T:\mathbb{S}^3\to\S^3$ such that $f=\varPhi\circ\varPi\circ T$, where
$\varPi$ is the
Hopf fibration and $\varPhi:\S^2(1/2)\to\S^2$ is the dilation from the
sphere $\S^2(1/2)$ of radius $1/2$.
\end{mythm}

A related classical problem is the Bernstein-type problem, which determines when a minimal submanifold
must be totally geodesic.
In the last years there has been a lot of research in this direction; for example, see \cite{assimos,ding,h-s1,h-s2,jost1,jost2,jost3,jost4,wangm,savas1,swx}.
A result in
the same spirit as our Theorem \ref{thmc} was obtained by Jost, Xin \& Yang \cite{jost5}. More precisely,
they proved that a minimal, coassociative, 4-dimensional graph with
constant singular values in $\real{7}$ either is flat or an open piece of the Lawson-Osserman
cone.

Let us mention a similarity of Theorem \ref{thmc} with the Chern conjecture \cite{chern,chern1}, which asks
for the classification of minimal submanifolds with constant norm of the second fundamental form
in spheres. In general, the conjecture is still open and is settled
only in the case of 3-dimensional hypersurfaces in $\S^4$ by the efforts of Peng \&
Terng \cite{peng1,peng2}, de Almeida \& Brito \cite{almeida}, and Chang \cite{chang}. In fact, it is shown that
a compact $3$-dimensional minimal hypersurface in $\S^4$ is either totally geodesic, a minimal Clifford torus, or
a minimal Cartan hypersurface. It is conjectured by Bryant that the same holds true even without the compactness
assumption; for more details see  \cite[page 524]{chang}.

There is an abundance of minimal maps from odd-dimensional spheres into the complex projective space that does not coincide
with the Hopf fibration. One easy way to construct such maps is by composing the Hopf fibration with a holomorphic automorphism of
the complex projective space; see for details Proposition \ref{comphol}. It turns out that these minimal maps are {\em weakly conformal}, i.e., the restriction of the
differential of $f$ on the orthogonal complement of its kernel is conformal. Hence, another interesting problem is to find conditions under which a minimal map $f:\S^{2n+1}\to\CP^n$ is
weakly conformal. In a matter of fact, motivated by a conjecture of Eells in harmonic map theory, we would expect that
any minimal map $f:\S^3\to\S^2$  must be weakly conformal; see \cite[Note 10.4.1, page 422]{baird1} and \cite[page 730]{wangg}.

Let us point out here that weakly conformal (not necessarily minimal) submersions of spheres with totally umbilical fibers were studied recently in \cite{heller,zawadzki,zawadzki2}.
According to a beautiful  result of Heller \cite[Theorem 3.7]{heller}, {\em up to conformal transformations of $\S^2$ and $\S^3$,
every conformal fibration of $\S^3$ by circles (not necessarily great circles) is the Hopf fibration.}

In the next theorem, we provide a partial answer to the problem of when a minimal submersions $f: \S^3\to\S^2$ is weakly conformal.
The quantity which will play a crucial role in our analysis is the {\em Gauss map
$\mathcal{G}$ of the
submersion $f$}. According to Baird \cite{baird}, $\mathcal{G}$ associates to each point $x\in \S^3$ the
line in the bundle $\mathbb{G}_1(\S^3)$ over $\S^3$, whose fiber at each point $x\in\S^3$ is the Grassmannian of 
oriented lines in $T_x\S^3$. Note that, the orthogonal complement $\mathcal{H}$ of the kernel of $df$ is a real
plane bundle and so it can be regarded as a complex
line bundle. By fixing at each point $x\in\S^3$ an oriented orthonormal basis of the singular value decomposition of
$df$ we may represent $\mathcal{H}_x$ as the direct sum $\operatorname{Re}(\mathcal{H}_x)\oplus
\operatorname{Im}(\mathcal{H}_x)$ of two
real lines, which we call {\em real} and {\em 
imaginary parts of $\mathcal{H}_x$, with respect to the given frame.} Note that at points where the non-zero singular 
values $\lambda_2\le\lambda_3$ of $df$ are distinct, the  choice of the aforementioned frame is unique.

\begin{mythm}\label{thma}
Let $f:\S^3\to\S^2$ be a minimal submersion whose Gauss map $\mathcal{G}$ (with respect to the graphical metric)
satisfies the condition
$$
(\lambda_3-\lambda_2)\big(\lambda_3\big|\operatorname{Im}d\mathcal{G}\big|^2-
\lambda_2\big|\operatorname{Re}d\mathcal{G}\big|^2\big)
\ge 0,
$$
where $0<\lambda_2\le\lambda_3$ are the non-zero singular values of the map $f$.
Then, $f$ is weakly conformal with totally geodesic fibers. Moreover, there exists an isometry $T:\S^3\to\S^3$ and a 
conformal diffeomorphism $\varPhi:\S^2\to\S^2$ such that
$f=\varPhi\circ\varPi\circ T,$
where $\varPi$ is the standard Hopf fibration.
\end{mythm}

Let $f:\S^3\to\S^2$ be a weakly conformal minimal submersion with totally geodesic fibers. Then, for
any choice of an orthonormal frame of the singular value decomposition we have that
$\big|\operatorname{Re}d\mathcal{G}\big|=\big|\operatorname{Im}d\mathcal{G}\big|.$
Moreover, the condition on Theorem \ref{thma}
can be equivalently expressed in terms of of the second fundamental form of the graph; see Section \ref{83928012021}.

Next, we turn our attention to equivariant minimal maps with respect to the standard isometric actions
$\S^1\times\S^1\hookrightarrow\S^3$ and $\S^1\hookrightarrow\S^2.$
It turns out that the minimality of such maps is expressed in terms of a degenerate
ODE; see Section \ref{equiv}. As an application of Theorem \ref{thma}, we deduce the following
uniqueness result.

\begin{mythm}\label{thmb}
Let $f:\mathbb{S}^3\to\mathbb{S}^2$ be an equivariant minimal submersion. Then $f$ is the composition of the
standard Hopf fibration with the dilation from the radius $1/2$ sphere into the unit sphere.
\end{mythm}

\section{Maps between Riemannian manifolds}
In this section, we briefly discuss the geometry of maps between Riemannian manifolds following
the notation in \cite{savas1,savas2,savas3,savas4}.

\subsection{Notation} Let $(M,\gm)$ and $(N,\gn)$ be two Riemannian manifolds of dimension $m$ and $n$, respectively.
The induced metric on $M\times N$ will be denoted by
$\gk=\gm\times \gn.$
Often we denote the product metric by $\langle\cdot\,,\cdot\rangle$. A smooth map
$f:M\to N$ defines an embedding $F:M\to M\times N$, by
$F(x)=(x,f(x)),$
for any $x\in M$ The {\em graph} of the map $f$ is defined to be the submanifold $\varGamma(f)=F(M)$. Since $F$ is an embedding,
it induces another Riemannian metric
$\gind=F^*\gk$
on $M$. The two natural projections
$\pi_{M}:M\times N\to M$ and $\pi_{N}:M\times N\to N$
are submersions,
that is they are smooth and have maximal rank. Note that the tangent bundle  of the product manifold
$M\times N$, splits as a direct sum
$T(M\times N)=TM\oplus TN.$
The four metrics $\gm,\gn,\gk$ and $\gind$ are related by
\begin{eqnarray}
\gk&=&\pi_M^*\gm+\pi_N^*\gn\,,\label{met1}\\
\gind&=&F^*\gk=\gm+f^*\gn\,.\label{met2}
\end{eqnarray}
The Levi-Civita connection $\nabla^{\gk}$ associated to $\gk$ is
related to the Levi-Civita connections $\nabla^{\gm}$ on $(M,\gm)$ and $\nabla^{\gn}$ on
$(N,\gn)$ by
$$\nabla^{\gk}=\pi_M^*\nabla^{\gm}\oplus\pi_N^*\nabla^{\gn}\,.$$
The corresponding curvature operator $\rk$ on  the product $M\times N$ is related to the curvature
operators $\rm$ on $(M,\gm)$ and $\rn$ on $(N,\gn)$ by
\begin{equation*}
\rk=\pi^{*}_{M}{R}_M\oplus\pi^{*}_{N}\rn.
\end{equation*}
Denote the Levi-Civita connection on $M$ with respect to the induced metric
$\gind=F^*\gk$ by $\nabla^{\gind}$ and the curvature tensor by $\rind$.
The differential $dF$ of the map $F$ is a section in $F^{\ast}T(M\times N)\otimes T^*M$.
The covariant derivative of it is called the \textit{second fundamental tensor}
$A$ of the graph. That is,
\begin{equation*}
A(v_1,v_2)=(\nabla^{F}_{v_1}dF)v_2=\nabla^{\gk}_{dF(v_1)}dF(v_2)-dF(\nabla^{\gind}_{v_1}v_2)
\end{equation*}
for $v_1,v_2\in \mathfrak{X}(M)$, where $\nabla^{F}$ is the induced connection on
$F^{\ast}T(M\times N)\otimes T^*M$
and $\nabla$ is the Levi-Civita connection associated with $\gind.$

The trace of $A$ with respect to the metric $\gind$ is called the
{\it mean curvature vector field} of $\varGamma(f)$ and it
will be denoted by
$$H=\operatorname{trace}A.$$
Note that $H$ is a section in the normal bundle of the
submanifold. If $H$ vanishes identically the graph is called
{\em minimal}. Following Schoen's \cite{schoen2} terminology, a map $f:M\to N$ is called
a \textit{minimal map} if its graph
$\varGamma(f)$ is a minimal submanifold of the product space.

By \textit{Gauss equation} the tensors $\rind$ and $\rk$
are related by the formula
$$
(\rind-F^*\rk)(v_1,v_2,v_3,v_4)=\langle A(v_1,v_3),A(v_2,v_4)\rangle
-\langle A(v_1,v_4),A(v_2,v_3)\rangle
$$
for any $v_1,v_2,v_3,v_4\in \mathfrak{X}(M)$. Moreover, the second fundamental form $A$ satisfies the
\textit{Codazzi equation}
$$
(\nabla^\perp_{v_1}A)(v_2,v_3)-(\nabla^{\perp}_{v_2}A)(v_1,v_3)=\big(\rk\bigl(dF(v_1),dF(v_2),dF(v_3)\bigr)\big)^{\perp}
$$
for any $v_1,v_2,v_3$ on $\mathfrak{X}(M)$.

\subsection{Singular value decomposition}\label{singval}
Consider the non-negative definite symmetric $2$-tensor $f^*\gn$ on $M$.
Fix now a point $x\in M$ and assume that $r=\rank df_x$. Obviously, $r\le\min\{m,n\}$.
Suppose further that at $x$ and with respect to $\gm$, the $2$-tensor $f^*{\gn}$ has $r$ zero eigenvalues
and $q$ distinct positive eigenvalues
$$\lambda^2_{1}\le\dots\le\lambda^2_q$$
with multiplicities
$m_{1},\dots,m_q$ and with eigenspaces $E_{1},\dots,E_q$, respectively. Hence,
$$m_{1}+\dots+m_q=r.$$
The corresponding values $\{0,\lambda_{1},\dots,\lambda_q\}$ are called {\em singular values} of the differential
$df$ of $f$.
At the point $x$ consider an orthonormal basis
$$\{\zeta_{s},\alpha_{i_1},\dots,\alpha_{i_q}\},$$
with respect to $\gm$ which diagonalizes $f^*\gn$, where $s\in\{1,\dots,m-r\}$,
$i_j\in\{1,\dots,m_j\}$ and $j\in\{1,\dots,q\}$. Note that $\{\zeta_1,\dots,\zeta_{m-r}\}$ span the {\em kernel}
$\mathcal{V}$ of $df$. Moreover, at the point $f(x)$ consider an orthonormal basis $\{\beta_{k},\beta_{i_1},\dots,\beta_{i_q}\}$
with respect to the metric $\gn$, where $k\in\{1,\dots,n-r\}$, $i_j\in\{1,\dots,m_j\}$ and $j\in\{1,\dots,q\}$,
such that
$$df(\zeta_s)=0,\,\,df(\alpha_{i_1})=\lambda_{1}\beta_{i_1},\dots,df(\alpha_{i_q})=\lambda_q\beta_{i_q}.$$
The above procedure is called the \textit{singular value decomposition}
of $df$.

It is a well-known fact that, with the above ordering, the singular values give rise to continuous functions.
In a matter
of fact, they are even smooth on an open and dense subset of $M$. In particular, they are smooth on
open subsets where the corresponding multiplicities are constant and the corresponding eigenspaces are
smooth distributions.

Now we are going to define a special basis for the tangent and the normal space of the graph
in terms of the singular values. The vectors
\begin{equation}\label{tangent}
\Big\{e_s=\zeta_s,e_{i_1}=\tfrac{\alpha_{i_1}}{\sqrt{1+\lambda^2_1}},\dots,e_{i_q}=\tfrac{\alpha_{i_q}}{\sqrt{1+\lambda^2_q}}\Big\},
\end{equation}
where $s\in\{1,\dots,m-r\}$,
$i_j\in\{1,\dots,m_j\}$ and $j\in\{1,\dots,q\}$,
form an orthonormal basis with respect to the metric $\gind$ of the tangent space $T_{x}M$. Moreover, the vectors
\begin{equation}\label{normal}
\Big\{\xi_k=\beta_k,\xi_{i_1}=\tfrac{-\lambda_1\alpha_{i_1}\oplus\beta_{i_1}}{\sqrt{1+\lambda^2_1}},\dots,
\xi_{i_q}=\tfrac{-\lambda_q\alpha_{i_q}\oplus\beta_{i_q}}{\sqrt{1+\lambda^2_q}}\Big\},
\end{equation}
where $k\in\{1,\dots,n-r\}$,
$i_j\in\{1,\dots,m_j\}$ and $j\in\{1,\dots,q\}$,
give an orthonormal basis of the normal space of the
graph $\varGamma(f)$ at $f(x)$.

We will discuss now the structure of the singular values of holomorphic maps $f:M\to N$, where $M$ and
$N$ are K\"ahler manifolds of the same dimension $2m$. Let us denote the complex structures
of $M$ and $N$ by $J_M$ and $J_N$, respectively. Then, $df\circ J_M=J_N\circ df$. Moreover, the product
manifold is again a K\"ahler manifold with complex structure given by
$$J_{M\times N}=J_M\oplus J_N.$$
It turns out that the metric $\gind$ is a K\"ahler metric for the complex structure $J_M$ of $M$. Moreover, the map $F:(M,\gind,J_M)\to (M\times N,\gk,J_{M\times N})$ given by
$$
F(x)=(x,f(x))
$$
is a holomorphic isometric immersion. Consequently, any holomorphic map between K\"ahler manifolds is minimal; see also
\cite{eells1}. Fix again a point $x\in M$. Since the map $f$ is holomorphic, there exists an orthonormal basis
$\{\alpha_i,J_M\alpha_i\}$, $i\in\{1,\dots,m\}$, of $T_xM$
with respect to $\gm$ and an orthonormal basis $\{\beta_i,J_N\beta_i\}$, $i\in\{1,\dots,m\}$, of $T_{f(x)}M$
with respect to $\gn$, such that
\begin{equation}\label{SKahler}
df(\alpha_i)=\lambda_i\beta_i \quad\text{and}\quad df(J_M\alpha_i)=\lambda_iJ_N\beta_i
\end{equation}
for any $i\in\{1,\dots,m\}$. Note that here we regard the singular values without taking into account their multiplicities. Then, the vectors
\begin{equation}\label{SKahler1}
\Big\{e_{i}=\tfrac{\alpha_{i}}{\sqrt{1+\lambda _{i}^{2} }},\,\,
e_{m+i}=\tfrac{J_M\alpha_{i}}{\sqrt{1+\lambda _{i}^{2} }}\Big\}
\end{equation}
where $i\in\{1,\dots,m\}$, form an orthonormal basis with respect to the metric $\gind$ of the tangent space $T_{x}M$ and the vectors
\begin{equation}\label{SKahler2}
\Big\{\xi_{i}=\tfrac{-\lambda_i\alpha_{i}\oplus\beta_i}{\sqrt{1+\lambda _{i}^{2} }},\,\,
\xi_{m+i}=\tfrac{-\lambda_iJ_M\alpha_{i}\oplus J_N\beta_i}{\sqrt{1+\lambda _{i}^{2} }}\Big\},
\end{equation}
where $i\in\{1,\dots,m\},$ form an orthonormal basis of the normal space of the graphical submanifold
$\varGamma(f)$ at $f(x)$.

\subsection{Second fundamental form of graphs}
We will see now how the second fundamental form $A$ of the graph $\varGamma(f)$
can be written in terms of the differential $df$ and the
{\em Hessian}
$$
B(v_1,v_2)=\nabla^{f}_{v_1}df(v_2)-df(\nabla^{\gm}_{v_1}v_2)
$$
of the map $f:(M,\gm)\to(N,\gn).$

In the following lemma, we see that the difference between the connections induced by $\gm$ and $\gind$ is given in terms of the singular values and the Hessian $B$ of the map
$f$.

\begin{lemma}\label{christoffel}
The following relation holds true
$$
\nabla^{\gind}_{v_1}v_2-\nabla^{\gm}_{v_1}v_2=\big(I_m+df^tdf\big)^{-1}df^tB(v_1,v_2)
$$
for any pair of vector fields $v_1,v_2\in\mathfrak{X}(M)$, where $df^t$ is the transpose of the
differential $df:(TM,\gm)\to(TN,\gn)$ and $I_m:TM\to TM$ the
identity transformation.
\end{lemma}
\begin{proof}
Consider a local coordinate system $\{\partial_1,\dots,\partial_m\}$ on $M$. From the Koszul formula and the symmetries
of $B$, we obtain that
\begin{eqnarray*}
2{\gind}\big(\nabla^{\gind}_{\partial_i}{\partial_j},\partial_k\big)\hspace{-6pt}
&-&\hspace{-6pt}2\,\gm\big(\nabla^{\gm}_{\partial_i}{\partial_j},\partial_k\big)\\
&=&\,\,\partial_i{\gind}(\partial_j,\partial_k)+\partial_j{\gind}(\partial_i,\partial_k)
-\partial_k{\gind}(\partial_i,\partial_j)\\
&&\,\,-\partial_i{\gm}(\partial_j,\partial_k)-\partial_j{\gm}(\partial_i,\partial_k)
+\partial_k{\gm}(\partial_i,\partial_j)\\
&=&\,\,\partial_i{\gn}(df(\partial_j),df(\partial_k))+\partial_j{\gn}(df(\partial_i),df(\partial_k))\\
&&\,\,-\partial_k{\gn}(df(\partial_i),df(\partial_j))\\
&=&\hspace{-6pt}2\,\gn\big(\nabla^{f^*TN}_{\partial_i}df(\partial_j),df(\partial_k)\big).
\end{eqnarray*}
By the definition of $B$ and the transpose of $df$, the above equation yields
\begin{eqnarray*}
&&\gm\big(df^tB(\partial_i,\partial_j),\partial_k\big)=\gn\big(B(\partial_i,\partial_j),df(\partial_k)\big)\\
&&\quad\quad\quad={\gind}\big(\nabla^{\gind}_{\partial_i}{\partial_j},\partial_k\big)
-\gm\big(\nabla^{\gm}_{\partial_i}{\partial_j},\partial_k\big)-\gn\big(df(\nabla^{\gm}_{\partial_i}\partial_j),df(\partial_k)\big)\\
&&\quad\quad\quad={\gind}\big(\nabla^{\gind}_{\partial_i}{\partial_j}-\nabla^{\gm}_{\partial_i}{\partial_j},\partial_k\big)\\
&&\quad\quad\quad=\gm\big(\nabla^{\gind}_{\partial_i}{\partial_j}-\nabla^{\gm}_{\partial_i}{\partial_j},\partial_k\big)
+\gn\big( df(\nabla^{\gind}_{\partial_i}{\partial_j}-\nabla^{\gm}_{\partial_i}{\partial_j}),df(\partial_k)\big)\\
&&\quad\quad\quad=\gm\big(\nabla^{\gind}_{\partial_i}{\partial_j}-\nabla^{\gm}_{\partial_i}{\partial_j},\partial_k\big)
+\gm\big(df^tdf(\nabla^{\gind}_{\partial_i}{\partial_j}-\nabla^{\gm}_{\partial_i}{\partial_j}),\partial_k\big)\\
&&\quad\quad\quad=\gm\big((I_m+df^tdf)(\nabla^{\gind}_{\partial_i}{\partial_j}-\nabla^{\gm}_{\partial_i}{\partial_j}),\partial_k\big).
\end{eqnarray*}
Since the difference between two connections is tensorial, we deduce that
$$
\nabla^{\gind}_{v_1}v_2-\nabla^{\gm}_{v_1}v_2=\big(I_m+df^tdf\big)^{-1}df^tB(v_1,v_2)
$$
for any pair $v_1,v_2\in\mathfrak{X}(M)$.
This completes the proof.
\end{proof}

In the next lemma, we express the second fundamental form of the graph in terms of the singular
values and the Hessian of $f$.

\begin{lemma}\label{Aform}
The second fundamental form $A$ of the graph is given by
$$
A(v_1,v_2)=\big(\nabla^{\gm}_{v_1}v_2-\nabla^{\gind}_{v_1}v_2,df(\nabla^{\gm}_{v_1}v_2-\nabla^{\gind}_{v_1}v_2)+B(v_1,v_2)\big)
$$
for any $v_1,v_2\in\mathfrak{X}(M)$. Equivalently, the second fundamental form $A$ can be written in the form
$$
A(v_1,v_2)=\big(-(I_m+df^tdf)^{-1}df^tB(v_1,v_2),(I_n+dfdf^t)^{-1}B(v_1,v_2)\big)
$$
where as usual $I_m:TM\to TM$ and $I_n:TN\to TN$ are the identity transformations.
\end{lemma}

\begin{proof}
Given $v_2\in\mathfrak{X}(M)$, we have
$dF(v_2)=(v_2,df(v_2)).$
Differentiating now with respect to $v_1\in\mathfrak{X}(M)$, we get
\begin{eqnarray*}
A(v_1,v_2)&=&\nabla^{F}_{v_1}dF(v_2)-dF(\nabla^{\gind}_{v_1}v_2)\\
&=&\big(\nabla^{\gm}_{v_1}v_2-\nabla^{\gind}_{v_1}v_2,
\nabla^{f}_{v_1}df(v_2)-df(\nabla^{\gind}_{v_1}v_2)\big)\\
&=&\big(\nabla^{\gm}_{v_1}v_2-\nabla^{\gind}_{v_1}v_2,df(\nabla^{\gm}_{v_1}v_2
-\nabla^{\gind}_{v_1}v_2)
+B(v_1,v_2)\big).
\end{eqnarray*}
Combining the last formula with Lemma \ref{christoffel}, we get the desired result.
\end{proof}

As an immediate consequence of Lemma \ref{Aform}, we obtain the
following well-known characterization of minimality of
graphs in terms of harmonicity; see for example \cite{eells1}.

\begin{corollary}
Suppose that $f:(M,\gm)\to (N,\gn)$ is a smooth map.
Then, the following two statements are equivalent:
\begin{enumerate}[\rm(a)]
\item The graph $\varGamma(f)$ is a minimal submanifold of product space $M\times N$.
\medskip
\item The trace of the Hessian $B$ of the map $f$ with respect to the metric  $\gind$ is zero, that is $\operatorname{tr}_{\gind}B=0.$
\end{enumerate}
Moreover, at points where $\operatorname{rank}df=\dim N$, the
minimality of the graph is equivalent with the vanishing of the trace, with respect to $\gind$, of the $(2,1)$-tensor
$W$ given by
$$
W(v_1,v_2)=\nabla^{\gm}_{v_1}v_2-\nabla^{\gind}_{v_1}v_2,
$$
for any $v_1,v_2\in \mathfrak{X}(M)$.
\end{corollary}
\begin{proof}
From Lemma \ref{Aform}, it follows that the mean curvature $H$ of the graph satisfies
$$
H=\operatorname{tr}_{\gind}A
=\big({\operatorname{tr}}_{\gind}W,df(\operatorname{tr}_{\gind}W)\big)+
\big(0,\operatorname{tr}_{\gind}B\big).
$$
Hence, the equivalence of $\rm(a)$ and $\rm(b)$ is immediate. Taking into account the second
expression in Lemma \ref{Aform} and the equations \eqref{tangent}
and \eqref{normal}, we deduce that at points where $\operatorname{rank}df=\dim N$, the
minimality of the graph is equivalent with the vanishing of the trace, with respect to the metric $\gind$, of the
$(2,1)$-tensor $W$.
This completes the proof.
\end{proof}

\section{Geometry of the Hopf fibrations}\label{DGHFalex}
According to Hurwitz's Theorem, there are precisely four normed division algebras over the real numbers;
the real numbers $\real{}$, the complex numbers $\C$, the quaternions $\mathbb{H}$, and the octonions $\mathbb{O}$.
Hopf \cite{hopf} showed that there natural submersions of unit euclidean spheres over corresponding projective spaces formed by $n$-dimensional
real normed division algebras. Here we explore the geometric properties of graphs generated from the Hopf
fibrations.

\subsection{Complex Hopf fibrations}
Let us regard the unit euclidean sphere $\mathbb{S}^{2n+1}$ as a subset of ${\C}^{n+1}$, with center at the origin, and denote by $J$ its standard
complex structure, i.e., left multiplication with $\bf i\in \C$. The vector field
$$\zeta=-Jp,$$
where $p$ is the position vector of $\mathbb{S}^{2n+1}$, is globally defined, and is called the {\em Reeb
vector field} of the sphere. Moreover, for any vector field $v$ on $\mathbb{S}^{n+1}$, the decomposition in tangent and
normal components determines a $(1,1)$-tensor field $\varphi$ and a $1$-form $\eta$ on $\mathbb{S}^{2n+1}$
such that
\begin{equation}\label{601}
Jv=\varphi(v)+\eta(v)p=\varphi(v)+\langle v,\zeta\rangle p,
\end{equation}
where $\langle\cdot\,,\cdot\rangle$ is the euclidean metric.
One can
easily check that $\varphi$ and $\eta$ satisfy the following properties:
$$
\varphi(\zeta)=0,\quad \eta \circ \varphi=0,\quad \eta(\zeta)=1\quad\text{and}\quad \varphi^{2}=-I+\eta \otimes \zeta,
$$
where $I$ is the identity transformation. Moreover, it holds
$$
\langle\varphi (v_1), \varphi(v_2)\rangle=\langle v_1, v_2\rangle-\eta(v_1)\eta(v_2),
$$
for $v_1, v_2 \in \mathfrak{X}(\mathbb{S}^{2n+1})$. Note that $(\eta,\zeta,\varphi)$ is the standard {\em contact structure} of $\S^{2n+1}$.

By a direct computation we see that for any tangent vector field on the sphere, we have
$$
\nabla^{\mathbb{S}^{2n+1}}_v\zeta=-D_vJp-\langle v,Jp\rangle p=-JD_vp+\langle Jv,p\rangle p
=-Jv+\eta(v) p.
$$
where $D$ is the euclidean Levi-Civita connection.
Hence, from \eqref{601} we get
$$
\varphi(v)=-\nabla^{\S^{2n+1}}_v\zeta,
$$
from where we deduce that the integral curves of the Reeb vector field are geodesics of the sphere.
Let now $\mathcal{V}$ be the sub-bundle generated by the Reeb vector field
and by $\mathcal{H}$ its orthogonal complement. $\mathcal{V}$ is called the {\em vertical bundle} and
$\mathcal{H}$ the {\em horizontal bundle.} Hence, for any $x\in\mathbb{S}^{2n+1}$, we have the orthogonal
decomposition
$$
T_{x}\mathbb{S}^{2n+1}=\mathcal{V}_{x}\oplus \mathcal{H}_{x}.
$$
Hence, tensor $\varphi$ is measuring how the distribution $\mathcal{H}$ is twisted within the tangent bundle of the sphere.

The {\em complex projective space} $\CP^n$ is defined as the set of equivalence classes of $\C^{n+1}-\{0\}$ under the equivalence
relation $\sim$ defined by $z \sim w$ if there is a nonzero element $\lambda\in\C$ that $z=\lambda w$. Equivalently, we may regard
$\CP^n$ as the quotient of $\S^{2n+1}$ under the group action of $\mathbb{S}^{1}\subset\C$ given
by
$$
(z_1,\dots,z_{n+1})e^{{\bf i}\theta}=(z_1e^{{\bf i}\theta},\dots,z_{n+1}e^{{\bf i}\theta}).
$$
The natural quotient map $f:\S^{2n+1}\to\CP^n$ is the so-called {\em Hopf fibration}. The map
$f$ is a submersion and clearly its fibers are great
circles. In a matter of fact, the kernel of the differential $df$ is spanned by the Reeb vector field.
We can endow $\CP^n$ with a Riemannian metric $\gind_{\mathbb{CP}^{n}}$, the so-called {\em Fubini-Study
metric}, which makes $f$ a Riemannian submersion. With this Riemannian metric, the complex projective space $\CP^n$ becomes a K\"ahler manifold with complex structure $J_{\CP^{n}}$ given by the formula
$$
df\circ\varphi=J_{\CP^{n}}\circ df.
$$
Combining a theorem due to Escobales \cite{escobales} with a result of Ucci \cite{ucci},
the complex Hopf fibration is {\em rigid} among all Riemannian
submersions with totally geodesic fibers in the following sense: {\em Let $g:\S^{2n+1}\to\CP^n$ be a Riemannian submersion
with totally geodesic fibers. Then there exists an isometry $T$ of $\S^{2n+1}$ such that $g\circ T=f.$}

Let us consider the graph over this Hopf fibration and explore its geometry.

\begin{proposition}
Let $f:\S^{2n+1}\to\CP^n$ be the complex Hopf fibration. Then, the following statements hold true:
\begin{enumerate}[\rm(a)]
\item
The second fundamental form $A$ of the graph $\varGamma(f)$ is zero along the kernel of $df$ as well as along
its orthogonal complement on the sphere.
\medskip
\item
The graph $\varGamma(f)$ is a minimal submanifold of the product $\S^{2n+1}\times\CP^n$ with squared norm
of the second fundamental form $|A|^2=n$.
\end{enumerate}
\end{proposition}
\begin{proof}
Recall that the kernel of $df$ is generated by the Reeb vector field $\zeta$. Moreover,
the singular values of $df$ are
$\lambda_{1}=0$ and $\lambda_{2}=\cdots=\lambda_{2n+1}=1$.
Using the Koszul formula and the fact that
$f$ is a Riemannian submersion we deduce that
for any pair of horizontal vector fields $v_1,v_2\in\mathcal{H}$ we have
\begin{equation}\label{B1}
B(v_1,v_2)=\nabla^{f}_{v_1}df(v_2)-df\big(\nabla^{\S^{2n+1}}_{v_1}v_2\big)=0,
\end{equation}
for details see also \cite[Lemma 4.5.1, page 119]{baird1}.
Hence, $B$ vanishes on the horizontal bundle. Since the Reeb vector field has geodesic integral
curves and it belongs to the kernel of $df$, we obtain that
\begin{equation}\label{B4}
B(\zeta,\zeta)=0.
\end{equation}
and
\begin{equation}\label{B5}
B(\zeta,v)=-df(\nabla^{\S^{2n+1}}_{v}\zeta)=df(\varphi(v)),
\end{equation}
for any horizontal vector $v$.

From Lemma \ref{Aform} now, we deduce that $A$ vanishes also in the horizontal bundle $\mathcal{H}$
as well as in the vertical bundle $\mathcal{V}$. Therefore, $\varGamma(f)$ is a minimal
submanifold. For the mixed terms, again from Lemma \ref{Aform}, we get that
\begin{equation}\label{AC1}
A(\zeta,v)=\frac{-\varphi(v)\oplus df(\varphi(v))}{2}.
\end{equation}
Recall that the restriction of $\varphi$ on the horizontal bundle is an isometry with respect
to $\gind_{\S^{2n+1}}$. Since $f$ is a Riemannian submersion, $\varphi$ is an isometry also with respect to
the graphical metric. Indeed, for any horizontal vectors $v_1$ and $v_2$ we have
\begin{eqnarray}\label{621}
\gind(\varphi(v_1),\varphi(v_2))&=&\gind_{\S^{2n+1}}(\varphi(v_1),\varphi(v_2))+
\gind_{\CP^n}(df(\varphi(v_1)),df(\varphi(v_2)))\nonumber\\
&=&\gind_{\S^{2n+1}}(v_1,v_2)+\gind_{\S^{2n+1}}(\varphi(v_1),\varphi(v_2))=2\gind_{\S^{2n+1}}(v_1,v_2)\nonumber\\
&=&\gind(v_1,v_2).
\end{eqnarray}
Let $\{\zeta;e_1,\dots,e_{n};e_{n+1}=\varphi(e_1),\dots,e_{2n}=\varphi(e_n)\}$ be a local orthonormal frame with respect to $\gind$, where
$\{e_1,\dots,e_{2n}\}$ span the horizontal space. Then, from \eqref{AC1} and \eqref{621} we obtain that
$$
|A|^{2}=2\sum_{i=1}^{2n}|A(\zeta,e_{i})|^{2}=n.
$$
This completes the proof.
\end{proof}

\begin{remark}
In the special case where $n=1$, the Hopf fibration is a minimal Riemannian submersion between spheres.
By scaling properly the target in order to become a unit sphere, we obtain a minimal submersion
$f:\mathbb{S}^3\to\mathbb{S}^2$ with constant singular values
$\lambda_1=0$ and $\lambda_2=2=\lambda_3.$
In this case, there exists an adapted orthonormal frame
$\{e_1=\zeta,e_2,e_3;\xi_2,\xi_3\},$
with respect to which the shape operators $A^{\xi_2}$
and $A^{\xi_3}$ are given by
$$
A^{\xi_2}=
  \left( {\begin{array}{ccc}
   0 & 0 & {2}/{5}\\
   0 & 0 & 0 \\
   {2}/{5} & 0 & 0\\
  \end{array} } \right)
  \quad\text{and}\quad
  A^{\xi_3}=
  \left( {\begin{array}{ccc}
   \,\,0 & -{2}/{5} & 0\\
   -{2}/{5} &\,\,0 & 0 \\
   \,\,0 & \,\,0 & 0\\
  \end{array} } \right),
$$
respectively. Moreover, the squared norm of the second fundamental form
is equal to
$$|A|^2=16/25.$$
\end{remark}

We show now that there are plenty of minimal maps $f:\S^{2n+1}\to\CP^n$
with totally geodesic fibers other than the complex Hopf fibration.

\begin{proposition}\label{comphol}
Let $f:\S^{2n+1}\to\mathbb{CP}^n$ be the complex Hopf fibration and
$g:\mathbb{CP}^n\to\mathbb{CP}^n$ a holomorphic map. Then, the composition
$G=g\circ f$ is again a minimal map with totally geodesic fibers. In the special case
$n=1$, the squared norm of the second fundamental form $A_G$ of the graph $\varGamma(G)$ is
$$|A_G|^2=4\frac{\lambda^2(1+\lambda^2)+ |\nabla^{\CP^{1}}\lambda|^2}{(1+\lambda^2)^3},$$
where $\lambda$ is the conformal factor of $g$.
\end{proposition}
\begin{proof}
At first, note that the Reeb vector field $\zeta$ is unit with respect to the graphical metric too. Moreover, the spaces
$\mathcal{V}$ and $\mathcal{H}$ are again perpendicular with respect to the graphical
metric. Furthermore, the differential of $G$ commutes with $\varphi$ and the complex
structure of $\CP^n$. From this observation, it follows that the restriction
of $\varphi$ on the horizontal bundle $\mathcal{H}$ is an isometry also with respect to the induced metric $\gind$.
 By a straightforward computation, the Hessian of $G$ is given by
\begin{equation}\label{Bcompose}
B_{G}(v_1,v_2)=B_g(df(v_1),df(v_2))+dg(B_f(v_1,v_2)),
\end{equation}
for any vector fields $v_1$ and $v_2$ of the sphere $\S^{2n+1}$. Hence,
$$
B_G(\zeta,\zeta)=B_g(df(\zeta),df(\zeta))+dg(B_f(\zeta,\zeta))=0.
$$
From Lemma \ref{Aform} we obtain that $A(\zeta,\zeta)=0$. Consequently,
$$
\nabla^{\gind}_{\zeta}\zeta=\nabla^{\S^{2n+1}}_{\zeta}\zeta=0,
$$
from where it follows that the integral curves of $\zeta$ are geodesics
also with respect to the graphical metric. Moreover, because the map $g$ is holomorphic,
$f$ is a Riemannian submersion, and \eqref{B1}, we have that
\begin{eqnarray*}
&&B_G(v,v)+B_G(\varphi(v),\varphi(v))=B_g(df(v),df(v))+B_g(df(\varphi(v)),df(\varphi(v)))\\
&&\quad\quad\quad\,\,\,\,\,\,=B_g(df(v),df(v))+B_g(J_{\CP^{n}}df(v),J_{\CP^{n}}df(v))\\
&&\quad\quad\quad\,\,\,\,\,\,=B_g(df(v),df(v))-B_g(df(v),df(v))\\
&&\quad\quad\quad\,\,\,\,\,\,=0,
\end{eqnarray*}
for any $v\in\mathcal{H}$. Consequently, from Lemma \ref{Aform} it follows that
$$
A_G(v,v)+A_G(\varphi(v),\varphi(v))=0,
$$
for any horizontal vector $v$. Since $\varphi$ satisfies $\varphi^2=-I_{2n}$ on $\mathcal{H}$, we can always
find a local orthonormal frame of the form
$$\{e_1=\zeta,\,e_i,\,e_{n+i}=\varphi(e_i)\}_{i\in\{2,\dots,n\}}$$
with respect to the graphical metric $\gind$. Hence,
$$
A_G(e_i,e_i)+A_G(e_{n+i},e_{n+i})=0,
$$
for any $i\in\{2,\dots,n+1\}$. Hence, the composition $G=g\circ f$ is again a minimal map.

Let us examine now the case where $f:\S^3\to\CP^1$ and $g:\CP^1\to\CP^1$ is a holomorphic map. In this case, the map
$g$ is conformal. Let us denote by $\lambda$ its conformal factor. Consider at a fixed point $x\in\S^3$
an orthonormal with respect to $\gind_{\S^{3}}$ frame $\{\alpha_1=\zeta,\alpha_2,\alpha_3=\varphi(\alpha_2)\}$ and at the point $f(x)$ an orthonormal with respect to $\gind_{\CP^{1}}$
frame $\{\beta_2,\beta_3\}$ that
diagonalizes the differential of $f$. Moreover, consider at $G(x)$ an orthonormal, with respect to
$\gind_{\CP^{1}}$, frame
$\{{\tilde\beta_2},{\tilde\beta}_3\}$ such that
$$
dg(\beta_2)=\lambda{\tilde\beta}_2\quad\text{and}\quad dg(\beta_3)=\lambda{\tilde\beta}_3.
$$
Note that the frames described above satisfy
$J_{\CP^{1}}\beta_2=\beta_3$ and $J_{\CP^{1}}{\tilde\beta}_2={\tilde\beta}_3.$
Moreover, the vectors
$$
e_1=\zeta,\quad e_2=\frac{\alpha_2}{\sqrt{1+\lambda^2}}\quad\text{and}\quad
e_3=\frac{\alpha_3}{\sqrt{1+\lambda^2}}
$$
an orthonormal frame with respect to the graphical metric. Since,
$$
dG(e_1)=0,\quad dG(e_2)=\frac{\lambda\,{\tilde\beta}_2}{\sqrt{1+\lambda^2}}\quad\text{and}\quad
dG(e_3)=\frac{\lambda\,{\tilde\beta}_3}{\sqrt{1+\lambda^2}},
$$
the map $G$ is weakly conformal. Assume that $\lambda$ is not identically zero, since otherwise we
have nothing to show. By  Koszul formula we get that, away from the zero set of $\lambda$,
it holds
\begin{eqnarray*}
B_g(\varepsilon_1,\varepsilon_2)=\varepsilon_1(\log\lambda)dg(\varepsilon_2)
+\varepsilon_2(\log\lambda)dg(\varepsilon_1)-\gind_{\CP^{1}}(\varepsilon_1,\varepsilon_2)dg\big(\nabla^{\CP^{1}}\log\lambda\big),
\end{eqnarray*}
where $\varepsilon_1$, $\varepsilon_2$ are tangent vectors of $\CP^1$. Hence,
from \eqref{Bcompose}, \eqref{B1}, \eqref{B4} and \eqref{B5} and we deduce that
$$
B_G(\alpha_1,\alpha_1)=0,\quad B_G(\alpha_1,\alpha_2)=\lambda\,{\tilde\beta}_3\quad\text{and}\quad
B_G(\alpha_1,\alpha_3)=-\lambda\,{\tilde\beta}_2.
$$
Furthermore,
$$
B_G(\alpha_2,\alpha_2)=\beta_2(\lambda)\tilde\beta_2-\beta_3(\lambda){\tilde\beta}_3
=-B_G(\alpha_3,\alpha_3)
$$
and
$$
B_G(\alpha_2,\alpha_3)=\beta_3(\lambda)\tilde\beta_2+\beta_2(\lambda){\tilde\beta}_3.
$$
From Lemma \ref{Aform}, it follows that
$$
A_G(e_1,e_1)=0,\,\, A_G(e_1,e_2)=\frac{\lambda\,\xi_3}{1+\lambda^2}\,\,\text{and}
\,\, A_G(e_1,e_3)=\frac{-\lambda\,\xi_2}{1+\lambda^2}.
$$
Moreover,
$$
A_G(e_2,e_2)=\frac{\beta_2(\lambda)\xi_2-\beta_3(\lambda)\xi_3}{(1+\lambda^2)^{3/2}}=-A_G(e_3,e_3)
$$
and
$$
A_G(e_2,e_3)=\frac{\beta_3(\lambda)\xi_2+\beta_2(\lambda)\xi_3}{(1+\lambda^2)^{3/2}}
$$
Therefore,
$$
|A_G|^2=\frac{4\lambda^2}{(1+\lambda^2)^2}+\frac{4|\nabla^{\CP^{1}}\lambda|^2}{(1+\lambda^2)^3}.
$$
This completes the proof.
\end{proof}

\begin{remark}
There is a plethora of surjective holomorphic maps
$g:\CP^n\to\CP^n$ which are not isometries. On the other hand, it is a well-known fact
in Algebraic Geometry that there are no non-constant holomorphic maps $g:\CP^n\to\CP^m$,
if $n>m$. Moreover, there are many examples of conformal diffeomorphisms of $\S^{2}=\CP^1$ 
which are not rotations. In fact, the subgroup of orientation preserving conformal diffeomorphisms of $\S^{2}$ 
is isomorphic with the {\em projective linear group} $PGL_{2}(\C),$ which contains $SO(3)$ as a maximal compact 
subgroup.
\end{remark}

\begin{remark}
We expect that any minimal map with totally geodesic fibers arise as the composition of
the Hopf fibration with a holomorphic map. This is the case where
the minimal map is already weakly conformal. Another interesting question is whether a
minimal map $f:\S^3\to\S^2$ with constant norm of the second fundamental form coincides
with the Hopf fibration.
\end{remark}

\subsection{Quaternionic Hopf fibrations}
As a vector space, the {\em quaternions} are
$$
\mathbb{Q}=\{a_0+a_1\,{\bf i}_1+a_2\,{\bf i}_2+a_3\,{\bf i}_3:a_0,a_1,a_2,a_3\in\real{4}\}.
$$
They become an associative algebra with $1$ as the multiplicative unit via
$$
{\bf i}_1^2={\bf i}_2^2={\bf i}_3^2=-1,\quad{\bf i}_1{\bf i}_2={\bf i}_3=-{\bf i}_2{\bf i}_1\quad\text{and cyclic permutations.}
$$
Consider now $\S^{4n+3}$ as a hypersurface of
$\mathbb{Q}^{n+1}\simeq\real{4n+4}$ with  center at the origin.
Let us also consider the complex structures
$J_1,J_2,J_3$ on $\real{4n+4}$, respectively, induced by the left multiplication for ${\bf i}_1, {\bf i}_2, {\bf i}_3$
in $\mathbb{Q}$. Each of these complex structures
when applied to the position
vector $p$ of the sphere gives a globally defined vector field
$$\zeta_i=-J_ip,$$
where $i\in\{1,2,3\}$, on $\S^{4n+3}$.
As in the complex case, we denote by $\eta_i$ the dual form associated with $\zeta_i$ and by $\varphi_i$
the $(1,1)$-tensor given by
$$
J_iv = \varphi_i(v) + \eta_i(v)p,
$$
for any $i\in\{1,2,3\}$ and $v\in\mathfrak{X}(\S^{4n+3})$. Clearly, each tensor $\varphi_i$ satisfies
$$
\varphi^2_i=-I_{4n+3}+\eta_i\otimes\zeta_i
$$
The pairs $\{\eta_i,\zeta_i,\varphi_i\}_{i\in\{1,2,3\}}$, give rise to the standard $3$-Sasakian structure on $\S^{4n+3}$.
The $1$-forms $\eta_{i}$ and the $(1,1)$-tensors
$\varphi_{i}$ are related through the identities
\begin{eqnarray}\label{Eq: equations3contact}
\varphi_{k}&=&\varphi_{i}\circ\varphi_{j}-\eta_{j}\otimes \zeta_{i}=-\varphi_{j}\circ\varphi_{i}+\eta_{i}\otimes \zeta_{j},\nonumber\\
 \zeta_{k}&=&\varphi_{i}(\zeta_{j})=-\varphi_{j}(\zeta_{i}),\\
\eta_{k}&=&\eta_{i}\circ\varphi_{j}=-\eta_{j}\circ \varphi_{i}\nonumber,
\end{eqnarray}
for an even permutation $(i, j, k)$ of $(1, 2, 3)$. Moreover, from the Weingarten formula we obtain
$$
\varphi_{i}(v)=-\nabla^{\mathbb{S}^{4n+3}}_{v}\zeta_{i},
$$
for any $i\in\{1,2,3\}$ and $v\in\mathfrak{X}(\S^{4n+3})$.
By using (\ref{Eq: equations3contact}), we get
\begin{equation}\label{Eq: connectionxiS4n+3}
\begin{array}{lll}
\nabla^{\mathbb{S}^{4n+3}}_{\zeta_{1}}\zeta_{1}=0, &\nabla^{\mathbb{S}^{4n+3}}_{\zeta_{1}}\zeta_{2}=
\zeta_{3}, & \nabla^{\mathbb{S}^{4n+3}}_{\zeta_{1}}\zeta_{3}=-\zeta_{2},\\
\\
\nabla^{\mathbb{S}^{4n+3}}_{\zeta_{2}}\zeta_{1}=-\zeta_{3}, &\nabla^{\mathbb{S}^{4n+3}}_{\zeta_{2}}\zeta_{2}=0,
& \nabla^{\mathbb{S}^{4n+3}}_{\zeta_{2}}\zeta_{3}=\zeta_{1}, \\
\\
\nabla^{\mathbb{S}^{4n+3}}_{\zeta_{3}}\zeta_{1}=\zeta_{2}, &\nabla^{\mathbb{S}^{4n+3}}_{\zeta_{3}}\zeta_{2}=
-\zeta_{1}, & \nabla^{\mathbb{S}^{4n+3}}_{\zeta_{3}}\zeta_{3}=0,
\end{array}
\end{equation}
and so the integral curves of the vector fields $\zeta_i$, $i\in\{1,2,3\}$, are geodesic circles.

The distribution
$\mathcal{V}$ formed by $\zeta_1,\zeta_2$ and $\zeta_3$ is integrable and its
leaves are geodesic $3$-spheres of $\S^{4n+3}$. As usual, vector fields on $\mathcal{V}$ are called
{\em vertical} and vector fields on the orthogonal complement $\mathcal{H}$ of $\mathcal{V}$ are
called {\em horizontal}.

The {\em quaternionic projective space} $\mathbb{QP}^n$ is defined as vectors in $\mathbb{Q}^{n+1}-\{0\}$
modulo left scalar multiplication. The space $\mathbb{QP}^n$ can be obtained also by
identifying the leaves of the vertical distribution $\mathcal{V}$ on the sphere $\S^{4n+3}$.
As in the complex case, we obtain a natural projection
$f:\S^{4n+3}\to\mathbb{QP}^n$ which can be made into a Riemannian submersion. The
 map $f$ is called {\em quaternionic Hopf fibration}. Equivalently, $\mathbb{QP}^n$ can be regarded
 as the set of orbits of the natural group action of $\S^3\subset\mathbb{Q}$ on the sphere
 $\S^{4n+3}\subset\mathbb{Q}^{n+1}$. Let us mention that $\mathbb{QP}^1$ with its canonical metric is isometric
to the round $4$-sphere of radius $1/2$.

\begin{proposition}
Let $f:\S^{4n+3}\to\mathbb{QP}^n$ be the quaternionic Hopf fibration. Then, the following
statements hold true:
\begin{enumerate}[\rm(a)]
\item
The second fundamental form $A$ of the graph $\varGamma(f)$ is zero along the kernel of $df$ as well as along
its orthogonal complement on the sphere.
\medskip
\item
The graph $\varGamma(f)$ is a minimal submanifold of the product $\S^{4n+3}\times\mathbb{QP}^n$ with squared norm
of the second fundamental form $|A|^2=6n$.
\end{enumerate}
\end{proposition}

\begin{proof}
As in the complex case, we can show that the spaces $\mathcal{V}$ and $\mathcal{H}$ are again perpendicular with respect to the graphical
metric $\gind$. Moreover, the restriction of each $\varphi_i$, $i\in\{1,2,3\}$, on $\mathcal{H}$ is an isometry with respect to both metrics
$\gind_{\S^{4n+3}}$ and $\gind$. Since $f$ is Riemannian submersion,
from the Koszul formula, it follows that the Hessian $B$ of $f$ vanishes on the horizontal bundle
$\mathcal{H}$. Additionally, since $\mathcal{V}$ form the kernel of $df$, from
\eqref{Eq: connectionxiS4n+3} we get
$$
B(\zeta_i,\zeta_j)=0\quad\text{and}\quad B(\zeta_i,v)=df(\varphi_i(v))
$$
for any $i,j\in\{1,2,3\}$ and $v\in\mathcal{H}$. According to Lemma \ref{Aform}, we obtain that $A$ vanishes
on $\mathcal{V}$ and $\mathcal{H}$ and that
\begin{equation}\label{AC2}
A(\zeta_i,v)=\frac{-\varphi_i(v)\oplus df(\varphi_i(v))}{2}
\end{equation}
for any $i\in\{1,2,3\}$ and $v\in\mathcal{H}$. Consider an orthonormal, with respect to
the graphical metric $\gind$, frame of the form
$$
\{e_i=\zeta_i,\,e_{l},\,\varphi_{i}(e_l)\}_{i\in\{1,2,3\};l\in\{4,\dots,n+3\}}
$$
on the sphere. Then, from \eqref{AC2} we obtain that
$$
|A|^2=6n.
$$
 This completes the proof.
\end{proof}

\subsection{Octonionic Hopf fibration}
Let us describe here the octonionic Hopf fibration.
Recall at first that as a vector space, the {\em octonions} $\mathbb{O}$ are described by
$$
\mathbb{O}=\big\{a_0{\bf i}_0+a_1{\bf i}_1+a_2{\bf i}_2+a_3{\bf i}_3+a_4{\bf i}_4+a_5{\bf i}_5+a_6{\bf i}_6+a_7{\bf i}_7:a_0,a_1,\dots,a_7\in\real{}\big\}.
$$
We can turn the space $\mathbb{O}$ of octonions into a nonassociative algebra with ${\bf i}_0$ as multiplicative unit. The standard canonical basis
$$\{{\bf i}_0,{\bf i}_1,\ldots,{\bf i}_7\}$$
for $\mathbb{O}$ satisfies the following multiplication rules:
\begin{equation*}
{\bf i}_m{\bf i}_n=\left\{
\begin{array}{ll}
{\bf i}_n, & \text{if}\,\,m=0,\\
\\
{\bf i}_m, &\text{if}\,\,n=0 ,\\
\\
-\delta_{mn}{\bf i}_0+\varepsilon_{mnk}{\bf i}_k, &\text{otherwise}
\end{array}
\right.
\end{equation*}
where $\varepsilon_{mnk}$ is the completely antisymmetric tensor with value $1$
when
$$(m,n,k)=(1,2,3),(1,4,5),(1,7,6),(2,4,6),(2,5,7),(3,4,7),(3,6,5).$$
For $n\ge 2$, the naive definition of $\mathbb{OP}^n$ as vectors in $\mathbb{O}^{n+1}-\{0\}$
modulo left scalar multiplication has the problem that the equality up to left scalar
multiplication fails to be an equivalence relation.

Consider now the unit sphere $\S^{15}$ as a hypersurface of $\mathbb{O}^{2}\simeq\real{16}$. Let us also consider the complex structures
$J_{1},J_{2},\ldots,J_{7}$ on $\real{16}$, respectively, induced by the left multiplication for ${\bf i}_1, {\bf i}_2,\ldots, {\bf i}_7$
in $\mathbb{O}$. Each of these complex structures when applied to the position
vector $p$ of the sphere gives a globally defined vector field
$$\zeta_i=-J_ip,$$
where $i\in\{1,2,\ldots,7\}$, on $\S^{15}$.
As in the complex and quaternionic case, we denote by $\eta_i$ the dual form associated with $\zeta_i$ and by $\varphi_i$
the $(1,1)$-tensor given by
$$
J_iv = \varphi_i(v) + \eta_i(v)p,
$$
for $i\in\{1,2,\ldots,7\}$ and $v\in\mathfrak{X}(\S^{15})$. Since $\mathbb{O}$ is a nonassociative algebra, 
the tensors $\varphi_i$ do not satisfy similar relations as in the quaternionic case.

The distribution
$\mathcal{V}$ formed by $\zeta_1,\zeta_2,\ldots,\zeta_7$ is integrable and its
leaves are geodesic $7$-spheres of $\S^{15}$. As usual, vector fields on $\mathcal{V}$ are called
{\em vertical} and vector fields on the orthogonal complement $\mathcal{H}$ of $\mathcal{V}$ are
called {\em horizontal}.

\begin{proposition}
Let $f:\S^{15}\to\mathbb{OP}^1=\S^8(1/2)$ be the octonionic Hopf fibration. Then, the following
statements hold true:
\begin{enumerate}[\rm(a)]
\item
The second fundamental form $A$ of the graph $\varGamma(f)$ is zero along the kernel of $df$ as well as along
its orthogonal complement on the sphere.
\medskip
\item
The graph $\varGamma(f)$ is a minimal submanifold of the product $\S^{15}\times\mathbb{S}^8(1/2)$ with squared norm
of the second fundamental form $|A|^2=28$.
\end{enumerate}
\end{proposition}
\begin{proof}
As in the quaternionic case, we can show that the spaces $\mathcal{V}$ and $\mathcal{H}$ are again perpendicular with respect to the graphical
metric $\gind$. Moreover, the restriction of each $\varphi_i$, $i\in\{1,2,\ldots,7\}$, on $\mathcal{H}$ is an isometry with respect to both metrics
$\gind_{\S^{15}}$ and $\gind$. Since $f$ is Riemannian submersion with totally geodesic
fibers, we get
$$
B(\zeta_i,\zeta_j)=0\quad\text{and}\quad B(\zeta_i,v)=df(\varphi_i(v))
$$
for any $i,j\in\{1,2,\ldots,7\}$ and $v\in\mathcal{H}$. According to Lemma \ref{Aform}, we obtain that $A$ vanishes
on $\mathcal{V}$ and $\mathcal{H}$ and that
\begin{equation*}
A(\zeta_i,v)=\frac{-\varphi_i(v)\oplus df(\varphi_i(v))}{2}
\end{equation*}
for any $i\in\{1,2,\ldots,7\}$ and $v\in\mathcal{H}$. Consider an orthonormal, with respect to
the graphical metric $\gind$, frame of the form
$\{\zeta_i,e_{1},e_{2},\dots,e_{8}\}_{i\in\{1,\dots,7\}}$
on $\S^{15}$, where $\{e_{1},\dots,e_8\}\in\mathcal{H}$. Then, we obtain that $|A|^2=28$.
\end{proof}
\section{Equivariant minimal maps}\label{equiv}
In this section, we describe another interesting class of maps from $\S^3$ to $\S^2$, generalizing Hopf's fibration from $\S^3$ into $\S^2$. Let us regard $\S^3$ as a subset of $\C\times\C$, i.e.,
$$
\S^3=\{(z_1,z_2)\in\C\times\C:|z_1|^2+|z_2|^2=1\}
$$
and $\S^2$ as a subset of $\C\times\real{}$, i.e.,
$$
\S^2=\{(z,\tau)\in\C\times\real{}:|z|^2+\tau^2=1\}.
$$
Recall that the standard action of $\S^1\times\S^1$ on $\S^3$ is given by
$$
(e^{i\theta_1},e^{i\theta_2})\cdot(z_1,z_2)=(e^{i\theta_1} \cdot z_1,e^{i\theta_2}\cdot z_2)
$$
and the standard action of $\S^1$ on $\S^2$ is given by
$$
e^{i\theta}\cdot(z,\tau)=(e^{i\theta}\cdot z,\tau).
$$
Define now the multiplication $\varrho_{kl}:\S^1\times \S^1\to\S^1$, $(k,l)\in\mathbb{Z}\times\mathbb{Z}$, given by
$$
\varrho_{kl}(e^{i\theta_1},e^{i\theta_2})=e^{i(k\theta_1+l\theta_2)}.
$$
A map $f:\S^3\to\S^2$ is called {\em equivariant with respect to $\varrho_{kl}$} if
$$
f\big((e^{i\theta},e^{i\theta_2})(z_1,z_2)\big)=\varrho_{kl}(e^{i\theta_1},e^{i\theta_2})f(z_1,z_2).
$$
Let see now how we can describe such equivariant maps. To this end,
we parametrize points in the sphere
$\S^3$ by $(\sin s\cdot e^{i\xi},\cos s\cdot e^{i\eta})$,
where $(\xi,\eta)\in(0,2\pi)\times(0,2\pi)$ and $s\in[0,\pi/2]$, and points in $\S^2$ by $(\sin a\cdot e^{i\sigma},\cos a)$,
where $\sigma\in(0,2\pi)$ and $a\in(0,\pi)$. Then, one can easily see that such an
equivariant map $f:\S^3\to\S^2$ can be represented in the form
\begin{equation}\label{Tequiv}
f_{kl}(\xi,\eta,s)=f(\sin s\cdot e^{i\xi},\cos s\cdot e^{i\eta})=(\sin a(s)\cdot e^{i(k\xi+l\eta)},\cos a(s)),
\end{equation}
where $a: [0,{\pi}/{2}] \mapsto [0, \pi]$ is a function satisfying the boundary conditions
\begin{equation*}
a(0)=0\quad\text{and}\quad a(\pi/2)=\pi.
\end{equation*}
The function $a$ is called {\em generating function}. Maps of the form \eqref{Tequiv} are also called
{\em $a$-Hopf constructions}; see for instance \cite[Chapter  X]{eells2}.
Note that in the case $a(s)=2s$, up to a dilation of $\S^2$, we obtain the Hopf fibration.

\begin{lemma}\label{singinv}
Let $f_{kl}:\S^3\to\S^2$ be an $a$-Hopf construction. Then, the singular values of $df_{kl}$
at a point $(\sin s\cdot e^{i\xi},\cos s\cdot e^{i\eta})$ are
$$\lambda_1=0,\quad \lambda_2=\sqrt{\tfrac{k^2{\sin}^2a}{{\sin}^{2}s}+\tfrac{l^2{\sin}^2a}{{\cos}^{2}s}} \quad\text{and}\quad \lambda_3=|a_s|.$$
In particular, the map $f_{kl}$ is weakly conformal if and only if
\begin{equation*}
a(s)=\left\{
\begin{array}{ll}
2\arctan(c\tan^ks), & \text{if}\,\,l=k,\\
\\
2\arctan\left(c\frac{\big(l\cosec\hspace{-1pt} s-\sqrt{l^2\cot^2\hspace{-2pt}s+k^2}\big)^l}
{\big(\sqrt{k^2{\tan}^2s+l^2}-k{\sec} s\big)^k}\right), &\text{if}\,\, l>k,
\end{array}
\right.
\end{equation*}
where $c$ is a positive constant and $s\in[0,\pi/2]$.
\end{lemma}
\begin{proof}
Note at first that with respect to the coordinate systems introduced above, the metrics of $\mathbb{S}^{3}$ and $\mathbb{S}^{2}$
 are given by the expressions
$$
\gind_{\mathbb{S}^{3}}={\sin}^{2}s\, d\xi^{2}+{\cos}^{2}s\, d\eta^{2}+d s^{2}\,\,\,\,\text{and}\,\,\,\,
\gind_{\mathbb{S}^{2}}={\sin}^{2}\tau\, d\sigma^{2}+d\tau^{2},
$$
respectively. One can easily verify that the vector fields
\begin{equation}\label{BS3}
v_{1}=\frac{\partial_\xi}{\sin s},\quad v_{2}=\frac{\partial_{\eta}}{\cos s}\quad\text{and}\quad v_{3}=\partial_s
\end{equation}
constitute an orthonormal basis of $T_{(\xi, \eta, s)}\mathbb{S}^{3}$. Similarly, the vector fields
\begin{equation}\label{BS2}
w_{1}=\frac{\partial_\sigma}{\sin \tau}\quad\text{and}\quad w_{2}=\partial_\tau
\end{equation}
constitute an orthonormal basis of $T_{(\sigma, \tau)}\mathbb{S}^{2}$.
Let us compute the differential of $f_{kl}$ now. We have,
\begin{eqnarray*}\label{Eq: differential varphi k, l}
df_{kl}(v_{1})&=&\frac{df_{kl}(\partial_\xi)}{\sin s}=\frac{k\partial_\sigma}{\sin s}=\frac{k\sin a}{\sin s}w_1,\\
df_{kl}(v_{2})&=&\frac{df_{kl}(\partial_\eta)}{\cos s}=\frac{l\partial_\sigma}{\cos s}=\frac{l\sin a}{\cos s}w_{1},
\end{eqnarray*}
and
\begin{equation}\label{122021154}
df_{kl}(v_{3})=df_{kl}(\partial_s)=a_s{\partial_\tau}=-a_s(-\partial_\tau).
\end{equation}
Diagonalizing $f_{kl}^*\gind_{\S^2}$ with respect to $\gind_{\S^3}$ we see that the squares of the singular values of the differential
$df_{kl}$ at an arbitrary point $\varPi(\xi,\eta,s)$ are
$$\lambda^2_1=0,\quad \lambda^2_2=\frac{k^2{\sin}^2a}{{\sin}^{2}s}+\frac{l^2{\sin}^2a}{{\cos}^{2}s} \quad\text{and}\quad\lambda^2_3=a^2_s$$
with eigendirections
\begin{equation}\label{SBS3}
\alpha_{1}=\frac{l (\sin s) v_{1}-k (\cos s) v_{2}}{\sqrt{l^{2}\sin^{2}s+k^{2}\cos^{2}s}},
\alpha_{2}=\frac{k (\cos s) v_{1}+l (\sin s) v_{2}}{\sqrt{l^{2}\sin^{2}s+k^{2}\cos^{2}s}},
\alpha_{3}=\partial_{s},
\end{equation}
respectively.

Taking into account that the values of $a$ are on $[0,\pi]$ that $a(0)=0$ and $a(\pi/2)=\pi$, we deduce that $f_{kl}$ is weakly conformal if and only if
$$a_s=\sin a\sqrt{\tfrac{k^2\cos^2s+l^2\sin^2s}{\cos^2 s\sin^2 s}}.$$
Integrating, it follows that $f_{kl}$ is weakly conformal if and only if
\begin{equation*}
a(s)=\left\{
\begin{array}{ll}
2\arctan(c\,{\tan}^ks), & \text{if}\,\,l=k,\\
\\
2\arctan\left(c\frac{\big(l\cosec\hspace{-1pt} s-\sqrt{l^2\cot^2\hspace{-2pt}s+k^2}\big)^l}
{\big(\sqrt{k^2\tan^2\hspace{-1pt}s+l^2}-k\sec\hspace{-1pt} s\big)^k}\right), &\text{if}\,\, l>k.
\end{array}
\right.
\end{equation*}
where $c$ is a positive constant and $s\in[0,\pi/2]$. This completes the proof.
\end{proof}

\begin{proposition}\label{secsing}
An $a$-Hopf construction $f_{kl}:\mathbb{S}^{3} \to \mathbb{S}^{2}$
is minimal if and only if the generating function $a:[0,\pi/2]\to[0,\pi]$ satisfies the equation
\begin{eqnarray}\label{Eq: ODE alpha}
0&=&\frac{a_{ss}}{1+a_s^{2}}
+\frac{\cos s\sin s\big({\cos}\, 2s+(l^{2}-k^{2})\sin^{2}a\big)}{\sin^{2}s\cos^{2}s+\sin^{2}a\big(l^{2}\sin^{2}s+k^{2}\cos^{2}s\big)}a_s\nonumber\\
&&\quad\quad\quad\quad\,-\frac{\sin a\cos a\big(k^{2}\cos^{2}s+l^{2}\sin^{2}s\big)}
{\sin^{2}s\cos^{2}s+\sin^{2}a\big(l^{2}\sin^{2}s+k^{2}\cos^{2}s\big)},
\end{eqnarray}
with boundary conditions $a(0)=0$ and $a(\pi/2)=\pi.$ Moreover, an $a$-Hopf construction $f_{kl}$
is weakly conformal if and only if $k^2=l^2$.
\end{proposition}
\begin{proof} Let us consider the frames $\{v_1,v_2,v_3\}$, $\{\alpha_1,\alpha_2,\alpha_3\}$ and $\{w_1,w_2\}$ as in \eqref{BS3} and
\eqref{SBS3} and \eqref{BS2},
respectively. Note that
\begin{equation*}
[v_{1}, v_{2}]=0,\,\,[v_{1}, v_{3}]=(\cot s)v_{1},\,\,[v_2,v_3]=-(\tan s) v_{2},
\end{equation*}
and
$$
[w_1,w_2]=(\cot \tau) w_{1}.
$$
From the Koszul formula and the above Lie brackets, we easily get that
\begin{equation*}\label{Eq: covariant derivative S3}
\begin{array}{lll}
\nabla^{\mathbb{S}^{3}}_{v_{1}}{v_{1}}=-(\cot s) v_{3}, &\nabla^{\mathbb{S}^{3}}_{v_{1}}{v_{2}}=0, &\nabla^{\mathbb{S}^{3}}_{v_{1}}{v_{3}}= (\cot s) v_{1},\\
\nabla^{\mathbb{S}^{3}}_{v_{2}}{v_{1}}=0, & \nabla^{\mathbb{S}^{3}}_{v_{2}}{v_{2}}=(\tan s) v_{3}, & \nabla^{\mathbb{S}^{3}}_{v_{2}}{v_{3}}=-(\tan s) v_{2},\\
\nabla^{\mathbb{S}^{3}}_{v_{3}}{v_{1}}=0, & \nabla^{\mathbb{S}^{3}}_{v_{3}}{v_{2}}=0, & \nabla^{\mathbb{S}^{3}}_{v_{3}}{v_{3}}=0,
\end{array}
\end{equation*}
and
\begin{equation*}\label{Eq: covariant derivative S2}
\nabla^{\mathbb{S}^{2}}_{w_{1}}{w_{1}}=-(\cot\tau) w_{2},\,  \nabla^{\mathbb{S}^{2}}_{w_{1}}{w_{2}}=(\cot \tau) w_{1},\,\,
\nabla^{\mathbb{S}^{2}}_{w_{2}}{w_{1}}=0,\,\, \nabla^{\mathbb{S}^{2}}_{w_{2}}{w_{2}}=0.
\end{equation*}
By a straightforward computation, we obtain the connection forms of the frame $\{\alpha_1,\alpha_2,\alpha_3\}$ with
respect to the Levi-Civita connection of $\S^3$, namely
\begin{equation*}\label{Eq: connectioneigen3sphere}
\begin{array}{ll}
\hspace{-6pt}\nabla^{\mathbb{S}^{3}}_{\alpha_{1}}\alpha_{1}=\dfrac{(k^{2}-l^{2})\sin s\cos s}{l^{2}\sin^{2}s+k^{2}\cos^{2}s}\alpha_{3},
&\hspace{-9pt}\nabla^{\mathbb{S}^{3}}_{\alpha_{1}}\alpha_{2}=-\dfrac{kl}{l^{2}\sin^{2}s+k^{2}\cos^{2}s}\alpha_{3},
\\
\hspace{-6pt}\nabla^{\mathbb{S}^{3}}_{\alpha_{2}}\alpha_{1} =
-\dfrac{kl}{l^{2}\sin^{2}s+k^{2}\cos^{2}s}\alpha_{3},
&\hspace{-9pt}\nabla^{\mathbb{S}^{3}}_{\alpha_{2}}\alpha_{2}=\dfrac{l^{2}\sin^{4}s-k^{2}\cos^{4}s}{\sin s\cos s(l^{2}\sin^{2}s+k^{2}\cos^{2}s)}\alpha_{3},
\\
\hspace{-6pt}\nabla^{\mathbb{S}^{3}}_{\alpha_{3}}\alpha_{1}=\dfrac{kl}{k^{2}\cos^{2}s+l^{2}\sin^{2}s}\alpha_{2},
&\hspace{-9pt}\nabla^{\mathbb{S}^{3}}_{\alpha_{3}}\alpha_{2}=-\dfrac{kl}{k^{2}\cos^{2}s+l^{2}\sin^{2}s}\alpha_{1},
\end{array}
\end{equation*}
and
\begin{eqnarray*}
\nabla^{\mathbb{S}^{3}}_{\alpha_{1}}\alpha_{3}&=&\frac{(l^{2}-k^{2})\cos s\sin s}{l^{2}\sin^{2}s+k^{2}\cos^{2}s}\alpha_{1}+\frac{k l}{l^{2}\sin^{2}s+k^{2}\cos^{2}s}\alpha_{2},\\
\nabla^{\mathbb{S}^{3}}_{\alpha_{2}}\alpha_{3}&=&\frac{kl}{k^{2}\cos^{2}s+l^{2}\sin^{2}s}\alpha_{1}
+\frac{k^{2}\cos^{4}s-l^{2}\sin^{4}s}{\sin s \cos s(l^{2}\sin^{2}s+k^{2}\cos^{2}a)}\alpha_{2},\\
\nabla^{\mathbb{S}^{3}}_{\alpha_{3}}\alpha_{3}&=&0.
\end{eqnarray*}
Now let us compute the second fundamental form of $\varGamma(f_{kl})$. According to \eqref{122021154},
it suffices to compute at points where $a_s>0$. Recall that, at a fixed point of the graph,
the vectors
$$
e_1=\alpha_1,\,\,e_2=\frac{\alpha_2}{\sqrt{1+\lambda_2^2}},\,\,e_3=\frac{\alpha_3}{\sqrt{1+\lambda_3^2}}
$$
forms an orthonormal basis of $\varGamma(f_{kl})$ and
$$
\xi_4=\frac{-\lambda_2\alpha_2\oplus w_1}{\sqrt{1+\lambda^2_2}},\,\,
\xi_5=\frac{-\lambda_3\alpha_3\oplus w_2}{\sqrt{1+\lambda^2_3}}
$$
forms an orthonormal basis in the normal bundle of the graph. From
$$
h_{ij}^{\alpha}=\langle\nabla^{F}_{e_i}dF(e_j),\xi_{\alpha}\rangle_{\S^3\times\S^2}
=\langle\nabla^{\S^3}_{e_i}e_j\oplus\nabla^{f_{kl}}_{e_i}df_{kl}(e_j),
\xi_{\alpha}\rangle_{\S^3\times\S^2}
$$
we get that
\begin{eqnarray*}
&&\hspace{-25pt}h_{11}^{4}=h_{12}^{4}=h_{22}^{4}=h_{33}^{4}=0,\\
&&\hspace{-25pt}h_{13}^{4}=\frac{-kl\lambda_{2}}{(l^{2}\sin^{2}s+k^{2}\cos^{2}s)\sqrt{(1+\lambda_{2}^{2})
(1+\lambda_3^{2})}},\\
&&\hspace{-25pt}h_{23}^{4}=\frac{\hspace{-14pt}a_s\cos a\sin s\cos s(k^{2}\cos^{2}s+l^{2}\sin^{2}s)+\sin a(l^{2}\sin^{4}s-k^{2}\cos^{4}s)}
{\hspace{-12pt}\sqrt{k^2\cos^2s+l^2\sin^2s}\big(\cos^{2}s\sin^{2}s+\sin^{2}a(k^{2}\cos^{2}s+l^{2}\sin^{2}s)\big)
\sqrt{1+\lambda_3^{2}}}.
\end{eqnarray*}
Additionally,
\begin{eqnarray*}
h_{11}^{5}&=&\frac{a_s(l^{2}-k^{2})\cos s\sin s}{\sqrt{1+\lambda_3^{2}}(l^2\sin^2s+k^2\cos^2s)},
\\
h_{12}^{5}&=&
\frac{k l a_s}{\sqrt{(1+\lambda_{2}^{2})(1+a_s^{2})}(l^2\sin^2s+k^2\cos^2s)},\\
h_{13}^5&=&0,\,\,h_{23}^{5}=0, \,\, h_{33}^{5}=\frac{a_{ss}}{(1+a_s^{2})\sqrt{1+\lambda^2_3}},
\end{eqnarray*}
and
$$
h_{22}^{5}=\frac{a_s\sin s\cos s(k^{2}\cos^{4}s-l^{2}\sin^{4}s)-\sin a\cos a(k^{2}\cos^{2}s+l^{2}\sin^{2}s)^{2}}{\sqrt{1+\lambda_3^{2}}(l^{2}\sin^{2}s+k^{2}\cos^{2}s)(\cos^{2}\sin^{2}s+\sin^{2}a(k^{2}\cos^{2}s+l^{2}\sin^{2}s))}.\nonumber
$$
Since
$$h_{11}^4=h_{22}^4=h_{33}^4=0$$
it follows that $f_{kl}$ is a minimal map if and only if
\begin{eqnarray*}
0&=&\frac{a_{ss}}{1+a_s^{2}}
+\frac{\cos s\sin s\big({\cos}\, 2s+(l^{2}-k^{2})\sin^{2}a\big)}{\sin^{2}s\cos^{2}s+\sin^{2}a\big(l^{2}\sin^{2}s+k^{2}\cos^{2}s\big)}a_s\\
&&\quad\quad\quad\quad\,-\frac{\sin a\cos a\big(k^{2}\cos^{2}s+l^{2}\sin^{2}s\big)}
{\sin^{2}s\cos^{2}s+\sin^{2}a\big(l^{2}\sin^{2}s+k^{2}\cos^{2}s\big)}.
\end{eqnarray*}
Let $f_{kl}$ be a weakly conformal minimal map. Then the function $a$ satisfies
\begin{equation}\label{Eq: ODE conformality}
a_s=\sin\hspace{-1pt} a \sqrt{\frac{k^{2}}{\sin^{2} s}+\frac{l^{2}}{\cos^{2} s}}.
\end{equation}
Differentiating (\ref{Eq: ODE conformality}), we get
\begin{equation}\label{Eq: second derivative alpha}
a_{ss}=\sin a\cos a \frac{k^{2}\cos^{2}s+l^{2}\sin^{2}s}{\sin^{2}s\cos^{2}s}
+\sin a \frac{\frac{l^{2}\sin s}{\cos^{3}s}-\frac{k^{2}\cos s}{\sin^{3}s}}{\sqrt{\frac{k^{2}}{\sin^{2}s}
+\frac{l^{2}}{\cos^{2}s}}}.
\end{equation}
Substituting (\ref{Eq: ODE conformality}) and (\ref{Eq: second derivative alpha}) in (\ref{Eq: ODE alpha}), we get
\begin{eqnarray*}
0&=&(l^{2}-k^{2})\Big(\frac{1}{4}-\frac{\cos^{2}2s}{4}+(k^{2}\cos^{2}s+l^{2}\sin^{2}s)\sin^{2}a\Big)\sin a\\
&=&(l^{2}-k^{2})\Big(\frac{\sin^{2}2s}{4}+(k^{2}\cos^{2}s+l^{2}\sin^{2}s)\sin^{2}a\Big)\sin a.
\end{eqnarray*}
Hence a weakly conformal map $f_{kl}$ is minimal only if $l^2=k^2$.
\end{proof}

\section{Structure equations of maps from $\S^3$ to $\S^2$}\label{83928012021}
We consider now maps $f:\S^{3}\to\S^2$ between unit euclidean spheres. Assume that $U$ is an
open neighbourhood of $\S^3$ where the kernel $\mathcal{V}$ of $df$ is a line bundle. Hence, in $U$
two singular values of $f$ are non-zero. For simplicity, let us denote them by $0<\lambda\le\mu$.
Fix $x \in U$ and consider an orthonormal basis $\{\alpha_{1}, \alpha_{2}, \alpha_{3}\}$ of $T_x\S^3$ with
respect to the spherical metric and an orthonormal basis $\{\beta_{2}, \beta_{3}\}$ of $T_{f(x)}\S^2$ again with respect to the
spherical metric such that
\begin{equation*}
\begin{array}{ccc}
d f(\alpha_{1})=0, & d f(\alpha_{2})=\lambda \beta_{2}, & d f(\alpha_{3})=\mu \beta_{3}.
\end{array}
\end{equation*}
Then, the vectors
$$e_{1}=\alpha_{1},\,\, e_{2}=\frac{\alpha_{2}}{\sqrt{1+\lambda^{2}}},\,\, e_{3}=\frac{\alpha_{3}}{\sqrt{1+\mu^{2}}}$$
form an orthonormal basis of $T_{x}\S^3$ with respect to $\gind$.  Moreover, the vectors
$$\xi_{4}=\frac{-\lambda \alpha _{2} \oplus \beta_{2}}{\sqrt{1+\lambda^{2}}}\quad\text{and}\quad
\xi_{5}=\frac{-\mu \alpha_{3}\oplus \beta_{3}}{\sqrt{1+\mu^{2}}}$$
constitute an orthonormal basis of the normal space of the graphical submanifold $\varGamma(f)$ at $f(x)$.
In the following, we adopt the following  terminology
\begin{equation*}
b_{ij}^{\alpha}=\langle B(\alpha_{i}, \alpha_{j}), \beta_{\alpha-2}\rangle_{\S^2}
\quad\text{and}\quad h_{ij}^{\alpha}=\langle A(e_{i}, e_{j}), \xi_{\alpha}\rangle_{\S^3\times\S^2}
\end{equation*}
for all $i, j\in\{1, 2, 3\}$ and $\alpha\in\{4, 5\}$.

\subsection{Jacobians and angles}
There are several quantities that encode information about the geometry of the graph
$F:U\subset\S^3\to\S^3\times\S^2$.
The first one is the Jacobian of the projection into the first factor. Namely, let $\Omega_{\S^3}$ be the volume form of $\S^3$. We
can extend to a parallel form on $\S^3\times\S^2$ by pulling it back via the natural projection map
$\pi_{\S^3}:\S^3\times\S^2\to\S^3$.
That is, we define $\Omega_1=\pi^*_{\S^3}\Omega_{\S^3}.$
Then, the Jacobian of the projection from the graph into $\S^3$ is $u_1=*(F^*\Omega_1)$
where $*$ stands for the Hodge star operator with respect to $\gind$. In terms of the singular values of the map $f$
the function $u_1$ has the form
$$
u_1=\frac{1}{\sqrt{(1+\lambda^2)(1+\mu^2)}}.
$$
There is another important function that will play a crucial role in our analysis.
Let us denote by $\mathcal{H}$ the orthogonal, with respect to the spherical metric, complement of $\mathcal{V}$
in $TU$. Observe that $\mathcal{H}$ is orthogonal to $\mathcal{V}$ also with respect to the graphical metric.
Since $\mathcal{H}$ is a two-dimensional sub-bundle it posses a complex structure $J_{\mathcal{H}}$. Let
$\Omega_{\S^2}$ the volume form $\S^2$. One can readily check that
the $2$-form $F^*\Omega_2$ is non-zero only on $\mathcal{H}$. Consequently, it must be a multiple of the volume
form $\Omega_{\mathcal{H}}$ of $(\mathcal{H},\gind)$. This means that there exists a smooth function $u_2$ on $U$ such that
$$
F^*\Omega_2=u_2\Omega_{\mathcal{H}}.
$$
Without loss of generality, we may assume that $u_2$ is positive on $U$. Then, in terms of the singular values of $f$ we have
$$
u_2=\frac{\lambda\mu}{\sqrt{(1+\lambda^2)(1+\mu^2)}}.
$$
Note that the function $w=u_1+u_2$ is less or equal than $1$. In particular, if at some point $x_0\in U$ it holds
$w(x_0)=1$ then $\lambda(x_0)=\mu(x_0)$. Hence, $w$ measures how much $f$ deviates from being
weakly conformal.
\subsection{The Gauss map of a submersion}
Following Baird \cite{baird}, the Gauss map $\mathcal{G}:(\S^{3},\gind) \mapsto \mathbb{G}_{1}(\S^{3})$ of the submersion
$f:(\S^{3},\gind) \mapsto \S^{2}$ associates to each point $x \in \S^{3}$ the line spanned by $e_{1}$, where $\mathbb{G}_{1}(\S^{3})$ 
denotes the Grassmann bundle over $\S^3$. Consider the tensor $\varphi: T\S^3 \mapsto \mathcal{H}\subset T\S^3$,
given by
$$\varphi(v)=-\nabla^{\gind}_{v}e_{1},$$
for all $v\in T\S^3$. The tensor $\varphi$ is (minus) the differential of the Gauss map $\mathcal{G}$ and
describes how the complex line bundle $\mathcal{H}$ is twisted within $T\S^3$. Fix now an oriented orthonormal frame $\{e_1,e_2,e_3=J_{\mathcal{H}}e_2\}$ of the singular value decomposition of
$df$. Then, with respect to this frame, we may introduce the tensors $\operatorname{Re}\varphi,\operatorname{Im}\varphi:T\S^3\to\mathcal{H}$ given by
$$
(\operatorname{Re}\varphi) v=\gind(\varphi (v),e_2)e_2 \quad\text{and}\quad (\operatorname{Im}\varphi) v=\gind(\varphi (v),J_{\mathcal{H}}e_2)J_{\mathcal{H}}e_2
$$
for any $v\in T\S^3$. Note that, at points where $\lambda\neq\mu$ the orthonormal frame $\{e_1,e_2,e_3\}$ is unique.

\subsection{Local formulas}
We will see now how all the notions defined above are related together.
\begin{lemma}\label{7.1}
Let $f:U\subset\S^3 \mapsto \S^2$ be a submersion with non-zero singular values $\lambda$ and $\mu$. Then,
on the open and dense subset of $U$ where $\lambda$ and $\mu$ are smooth, we have
\begin{equation*}
\begin{array}{lll}
b_{11}^{4}=\sqrt{1+\lambda^{2}}\,h_{11}^{4}, & b_{12}^{4}=\alpha_{1}(\lambda), & b_{13}^{4}=\sqrt{(1+\lambda^2)(1+\mu^2)}\,h_{13}^{4}, \\
b_{22}^{4}=\alpha_{2}(\lambda), & b_{23}^{4}=\alpha_{3}(\lambda), & b_{33}^{4}=(1+\mu^{2})\sqrt{1+\lambda^{2}}\,h_{33}^{4},
\end{array}
\end{equation*}
\begin{equation*}
\begin{array}{lll}
b_{11}^{5}=\sqrt{1+\mu^{2}}\, h_{11}^{5}, &b_{13}^{5}=\alpha_{1}(\mu), &  b_{12}^{5}=\sqrt{(1+\lambda^2)(1+\mu^2)}\,h_{12}^{5}, \\
b_{33}^{5}=\alpha_{3}(\mu), & b_{23}^{5}=\alpha_{2}(\mu), & b_{22}^{5}=(1+\lambda^{2})\sqrt{1+\mu^{2}}\,h_{22}^{5},
\end{array}
\end{equation*}
and
\begin{equation}\label{Eq: components h}
\begin{array}{lll}
h_{12}^{4}=e_{1}(\arctan \lambda), & h_{22}^{4}=e_{2}(\arctan \lambda), & h_{23}^{4}=e_{3}(\arctan \lambda),\\
h_{13}^{5}=e_{1}(\arctan \mu), & h_{23}^{5}=e_{2}(\arctan \mu), & h_{33}^{5}=e_{3}(\arctan \mu),
\end{array}
\end{equation}
where all indices are with respect to the orthonormal frames of the singular value decomposition.
\end{lemma}
\begin{proof}
Differentiating with respect to $\alpha_1$ the identity
$$\langle d f(\alpha_{2}), d f(\alpha_{2})\rangle_{\S^2}=\lambda^{2},$$
and making use that $\alpha_1$ spans the kernel of $df$, we get
\begin{eqnarray*}\label{Eq: B12 4}
\lambda \alpha_{1}(\lambda)&=&\langle B(\alpha_{1}, \alpha_{2})+df(\nabla_{\alpha_{1}}^{\S^3}\alpha_{2}), d f(\alpha_{2})\rangle_{\S^2}\\
&=&\langle B(\alpha_{1}, \alpha_{2}), d f(\alpha_{2})\rangle_{\S^2}+\langle \nabla_{\alpha_{1}}^{\S^3}\alpha_{2}, \alpha_{3}\rangle_{\S^3}
\langle d f(\alpha_{2}), df(\alpha_{3})\rangle_{\S^2}\\
&=&\lambda\, b^4_{12}.
\end{eqnarray*}
Therefore,
\begin{equation}\label{Eq: B12 4}
b^4_{12}=\alpha_1(\lambda).
\end{equation}
Similarly, differentiating with respect to $\alpha_{1}$, $\alpha_2$ and $\alpha_3$ the equations
$$\langle d f(\alpha_{2}), d f(\alpha_{2})\rangle_{\S^2}=\lambda^{2}\quad\text{and}\quad\langle d f(\alpha_{3}), d f(\alpha_{3})\rangle_{\S^2}=\mu^{2} $$
we obtain
\begin{equation}\label{Eq: B13 5}
b_{13}^{5}=\alpha_{1}(\mu),\,\,b^{4}_{22}=\alpha_2(\lambda),\,\, b_{23}^4=\alpha_3(\lambda),\,\,b_{23}^{5}=\alpha_{2}(\mu),\,\, b_{33}^{5}=\alpha_{3}(\mu).
\end{equation}
From Lemma \ref{Aform}, equations \eqref{Eq: B12 4} and \eqref{Eq: B13 5} we get the expression relating the second fundamental form
of the graph with the singular values and the Hessian of the map $f$.
This completes the proof.
\end{proof}
\begin{lemma}\label{7.2}
Let $f:U\subset\S^{3} \mapsto \S^{2}$ be a smooth map with non-zero distinct singular values $\lambda$ and $\mu$. Then,
\begin{eqnarray*}
\nabla_{\alpha_{1}}^{\S^3}\alpha_{1}&=&-\frac{\sqrt{1+\lambda^{2}}}{\lambda}h_{11}^{4}\alpha_{2}
-\frac{\sqrt{1+\mu^{2}}}{\mu}h_{11}^{5}\alpha_{3},\\
\nabla_{\alpha_{1}}^{\S^3}\alpha_{2}&=&\,\,\,\,\frac{\sqrt{1+\lambda^{2}}}{\lambda}h_{11}^{4}\alpha_{1}
-\frac{\mu h_{12}^{5}+\lambda h_{13}^{4}}{\big(\mu^2-\lambda^{2}\big)u_1}\alpha_{3},\\
\nabla_{\alpha_{1}}^{\S^3}\alpha_{3}&=&\,\,\,\,\frac{\sqrt{1+\mu^{2}}}{\mu}h_{11}^{5}\alpha_{1}+\frac{\mu h_{12}^{5}
+\lambda h_{13}^{4}}{\big(\mu^{2}-\lambda^{2}\big)u_1}\alpha_{2}.
\end{eqnarray*}
Moreover,
\begin{eqnarray*}
\nabla_{\alpha_{2}}^{\S^3}\alpha_{1}&=&-\alpha_{1}(\log\lambda)\alpha_{2}-
\frac{h_{12}^{5}}{\mu u_1}\alpha_{3}\\
\nabla_{\alpha_{2}}^{\S^3}\alpha_{2}&=&\,\,\,\,\alpha_{1}(\log\lambda)\alpha_{1}-
\frac{u_1\lambda\alpha_{3}(\lambda)+\mu\sqrt{1+\lambda^{2}}h_{22}^{5}}
{\big(\mu^2-\lambda^{2}\big)u_1}\alpha_{3},\\
\nabla_{\alpha_{2}}^{\S^3}\alpha_{3}&=&\,\,\,\,\frac{h_{12}^{5}}{\mu u_1}\alpha_{1}+
\frac{u_1\lambda\alpha_{3}(\lambda)+\mu\sqrt{1+\lambda^{2}}h_{22}^{5}}
{\big(\mu^{2}-\lambda^{2}\big)u_1}\alpha_{2}.
\end{eqnarray*}
Additionally,
\begin{eqnarray*}
\nabla_{\alpha_{3}}^{\S^3}\alpha_{1}&=&-\frac{h_{13}^{4}}{\lambda u_1}\alpha_{2}-\alpha_{1}(\log\mu)\alpha_{3}\\
\nabla_{\alpha_{3}}^{\S^3}\alpha_{2}&=&\,\,\,\,\frac{h_{13}^{4}}{\lambda u_1}\alpha_{1}
-\frac{u_1\mu \alpha_{2}(\mu)
+\lambda\sqrt{1+\mu^{2}}h_{33}^{4}}{\big(\mu^2-\lambda^{2}\big)u_1}\alpha_{3},\\
\nabla_{\alpha_{3}}^{\S^3}\alpha_{3}&=&\,\,\,\,\alpha_{1}(\log\mu)\alpha_{1}
+\frac{u_1\mu \alpha_{2}(\mu)+\lambda\sqrt{1+\mu^{2}}h^4_{33}}{\big(\mu^{2}-\lambda^{2}\big)u_1}\alpha_{2}.
\end{eqnarray*}
\end{lemma}
\begin{proof}
By a straightforward computation, we have
\begin{eqnarray*}
h_{11}^{4}&=&\langle \nabla^F_{e_{1}}d F(e_{1}), \xi_{4}\rangle_{\S^3\times\S^2}
=\frac{1}{\sqrt{1+\lambda^2}}\langle\nabla^{\S^3}_{\alpha_1}\alpha_1\oplus 0,-\lambda\alpha_2\oplus\beta_2 \rangle_{\S^3\times\S^2}\\
&=&-\frac{\lambda}{\sqrt{1+\lambda^{2}}}\langle\nabla_{\alpha_{1}}^{\gm}\alpha_{1}, \alpha_{2}\rangle_{\S^3}.
\end{eqnarray*}
Similarly, we deduce
$$
h_{11}^{5}=-\frac{\mu}{\sqrt{1+\mu^{2}}}\langle\nabla_{\alpha_{1}}^{\S^3}\alpha_{1}, \alpha_{3}\rangle_{\S^3}.
$$
As a consequence, we get
\begin{equation}\label{Eq: nabla alpha1 alpha1}
\langle\nabla_{\alpha_{1}}^{\S^3}\alpha_{1}, \alpha_{2}\rangle_{\S^3}=-\frac{\sqrt{1+\lambda^{2}}}{\lambda}h_{11}^{4}\,\, \text{and}\,\,
 \langle\nabla_{\alpha_{1}}^{\S^3}\alpha_{1}, \alpha_{3}\rangle_{\S^3}=-\frac{\sqrt{1+\mu^{2}}}{\mu}h_{11}^{5}.
\end{equation}
Using (\ref{Eq: B12 4}) and (\ref{Eq: B13 5}), we have
\begin{eqnarray*}
\alpha_{1}(\lambda)&=&\langle B(\alpha_{1}, \alpha_{2}), \beta_{2}\rangle_{\S^2}\\
&=&\langle\nabla_{\alpha_{2}}^{f}d f(\alpha_{1})-d f(\nabla_{\alpha_{2}}^{\S^3}\alpha_{1}), \beta_{2}\rangle_{\S^2}
=-\langle d f(\nabla_{\alpha_{2}}^{\S^3}\alpha_{1}), \beta_{2}\rangle_{\S^2}\\
&=&-\lambda\langle\nabla_{\alpha_{2}}^{\S^3}\alpha_{1}, \alpha_{2}\rangle_{\S^2},
\end{eqnarray*}
and so
\begin{equation}\label{Eq: nabla alpha2 alpha1}
\langle\nabla_{\alpha_{2}}^{\S^3}\alpha_{1}, \alpha_{2}\rangle_{\S^3}=-\alpha_{1}(\log\lambda).
\end{equation}
Similarly, we get
\begin{equation}\label{Eq: nabla alpha3 alpha1}
\langle\nabla_{\alpha_{3}}^{\S^3}\alpha_{1}, \alpha_{3}\rangle_{\S^3}=-\alpha_{1}(\log\mu).
\end{equation}
Furthermore,
\begin{eqnarray}\label{Eq: nabla alpha32 alpha 1a}
h_{12}^{5}&=&\langle\nabla^F_{e_{2}}d F(e_{1}), \xi_{5}\rangle_{\S^3\times\S^2}\nonumber\\
&=&-\frac{1}{\sqrt{(1+\lambda^2)(1+\mu^2)}} \langle\nabla^{\S^3}_{\alpha_{2}}\alpha_{1}\oplus 0, -\mu \alpha_{3}\oplus \beta_{3}\rangle_{\S^3\times\S^2}
\nonumber\\
&=&-\frac{\mu}{\sqrt{(1+\lambda^2)(1+\mu^2)}} \langle\nabla_{\alpha_{2}}^{\S^3}\alpha_{1}, \alpha_{3}\rangle_{\S^3}
\end{eqnarray}
and in the same way we obtain
\begin{equation}\label{Eq: nabla alpha32 alpha 1b}
h_{13}^{4}=-\frac{\lambda}{\sqrt{(1+\lambda^2)(1+\mu^2)}} \langle\nabla_{\alpha_{3}}^{\S^3}\alpha_{1}, \alpha_{2}\rangle_{\S^3}.
\end{equation}
Differentiating with respect to $\alpha_{1}$ the identity
$$\langle d f(\alpha_{2}), d f(\alpha_{3})\rangle_{\S^2}=0,$$
we get
\begin{eqnarray*}
&&\hspace{-22pt}0=\langle\nabla_{\alpha_{1}}^{f}d f(\alpha_{2}), d f(\alpha_{3})\rangle_{\S^2}
+\langle d f(\alpha_{2}), \nabla_{\alpha_{1}}^{f}d f(\alpha_{3})\rangle_{\S^2}\\
&&\hspace{-14pt}=\langle B(\alpha_{1}, \alpha_{2})+d f(\nabla_{\alpha_{1}}^{\S^3}\alpha_{2}), d f(\alpha_{3})\rangle_{\S^2}
\hspace{-2pt}+\hspace{-3pt}\langle d f(\alpha_{2}), B(\alpha_{1}, \alpha_{3})+d f(\nabla_{\alpha_{1}}^{\S^3}\alpha_{3})\rangle_{\S^2}\\
&&\hspace{-14pt}=\mu b_{12}^{5}+\lambda b_{13}^{4}+(\mu^{2}-\lambda^{2})\langle\nabla^{\S^3}_{\alpha_{1}}\alpha_{2}, \alpha_{3}\rangle_{\S^3}.
\end{eqnarray*}
Using the expressions of Lemma \ref{7.1}, we obtain
\begin{equation}\label{Eq: connection 123}
\langle\nabla^{\S^3}_{\alpha_{1}}\alpha_{2}, \alpha_{3}\rangle_{\S^3}=-\sqrt{(1+\lambda^2)(1+\mu^2)}\frac{\mu h_{12}^{5}
+\lambda h_{13}^{4}}{\mu^{2}-\lambda^2}.
\end{equation}
Differentiating with respect to $\alpha_{2}$ and $\alpha_3$ the identity
$$\langle d f(\alpha_{2}), d f(\alpha_{3})\rangle_{\S^2}=0$$
and proceeding in the same way, we get
\begin{equation}\label{Eq: connection 223}
\langle\nabla^{\S^3}_{\alpha_{2}}\alpha_{2}, \alpha_{3}\rangle_{\S^3}
=\frac{\lambda \alpha_{3}(\lambda)+\mu(1+\lambda^{2})\sqrt{1+\mu^{2}}h_{22}^{5}}{\lambda^{2}-\mu^{2}}
\end{equation}
and
\begin{equation}\label{Eq: connection 332}
\langle\nabla^{\S^3}_{\alpha_{3}}\alpha_{3}, \alpha_{2}\rangle_{\S^3}=\frac{\mu \alpha_{2}(\mu)+\lambda(1+\mu^{2})\sqrt{1+\lambda^{2}}h_{33}^{4}}{\mu^{2}-\lambda^{2}}.
\end{equation}
Combining (\ref{Eq: nabla alpha1 alpha1}), (\ref{Eq: nabla alpha2 alpha1}), (\ref{Eq: nabla alpha3 alpha1}), (\ref{Eq: nabla alpha32 alpha 1a}),
\eqref{Eq: nabla alpha32 alpha 1b}, (\ref{Eq: connection 123}), (\ref{Eq: connection 223}) and (\ref{Eq: connection 332}), we get
the desired result.
\end{proof}
\begin{lemma}\label{7.3}
Let $f:U\subset\S^{3} \mapsto \S^{2}$ be a map with non-zero distinct singular values $\lambda$ and
$\mu$. Then, the normal connection of the graphical submanifold $\varGamma(f)$ is given by
\begin{eqnarray}
\nabla^{\perp}_{e_{1}}\xi_{4}&=&\frac{\lambda(1+\mu^{2})h_{12}^{5}+\mu(1+\lambda^{2})
h_{13}^{4}}{\lambda^{2}-\mu^{2}}\xi_{5}, \nonumber\\
\nabla^{\perp}_{e_{2}}\xi_{4}&=&\frac{\mu e_{3}(\lambda)+\lambda(1+\mu^{2})h_{22}^{5}}{\lambda^{2}-\mu^{2}}\xi_{5}, \nonumber\\
\nabla^{\perp}_{e_{3}}\xi_{4}&=&\frac{\lambda e_{2}(\mu)+\mu(1+\lambda^{2})h_{33}^{4}}{\lambda^{2}-\mu^{2}}\xi_{5},\nonumber
\end{eqnarray}
where all the frames are with respect to the singular value decomposition.
\end{lemma}
\begin{proof}
By a straightforward computation, we obtain
\begin{eqnarray*}
\langle\nabla^{\perp}_{e_{1}}\xi_{4}, \xi_{5}\rangle_{\S^3\times\S^2}
&=&\langle\nabla^{F}_{\alpha_{1}}\frac{-\lambda \alpha_{2}\oplus \beta_{2}}{\sqrt{1+\lambda^{2}}},
\frac{-\mu \alpha_{3} \oplus \beta_{3}}{\sqrt{1+\mu^{2}}}\rangle_{\S^3\times\S^2}\\
&=&\frac{\langle\nabla^F_{\alpha_{1}}(-\lambda\alpha_{2} \oplus \beta_{2}), -\mu \alpha_{3}\oplus \beta_{3}\rangle_{\S^3\times\S^2}}
{\sqrt{(1+\lambda^2)(1+\mu^2)}}\\
&=&\frac{\lambda \mu \langle\nabla^{\S^3}_{\alpha_{1}}\alpha_{2}, \alpha_{3}\rangle_{\S^3}
+\lambda^{-1}\langle\nabla^{f}_{\alpha_{1}}d f(\alpha_{2}), \beta_{3}\rangle_{\S^2}}
{\sqrt{(1+\lambda^2)(1+\mu^2)}}\\
&=& \frac{\lambda \mu \langle\nabla^{\S^3}_{\alpha_{1}}\alpha_{2}, \alpha_{3}\rangle_{\S^3}+\lambda^{-1}b_{12}^{5}
+\lambda^{-1}\mu\langle\nabla^{\S^3}_{\alpha_{1}}\alpha_{2}, \alpha_{3}\rangle_{\S^3}}{{\sqrt{(1+\lambda^2)(1+\mu^2)}}}.
\end{eqnarray*}
Using the formulas from Lemma \ref{7.1} and Lemma \ref{7.2}, we get
\begin{eqnarray*}
\langle\nabla^{\perp}_{e_{1}}\xi_{4}, \xi_{5}\rangle_{\S^3\times\S^2}
&=& \frac{\lambda \mu \langle\nabla^{\S^3}_{\alpha_{1}}\alpha_{2}, \alpha_{3}\rangle_{\S^3}+\lambda^{-1}b_{12}^{5}
+\lambda^{-1}\mu\langle\nabla^{\S^3}_{\alpha_{1}}\alpha_{2}, \alpha_{3}\rangle_{\S^3}}{{\sqrt{(1+\lambda^2)(1+\mu^2)}}}\\
&=&-\lambda\mu \frac{\mu h_{12}^{5}+\lambda h_{13}^{4}}{\mu^2-\lambda^{2}}
+\frac{h_{12}^{5}}{\lambda}-\frac{\mu(\mu h_{12}^{5}+\lambda h_{13}^{4})}{\lambda(\mu^2-\lambda^{2})}\\
&=&-\frac{\lambda(1+\mu^{2})h_{12}^{5}+\mu(1+\lambda^{2})h_{13}^{4}}{\mu^2-\lambda^{2}}.
\end{eqnarray*}
Exactly in the same way, we compute the other terms. This completes the proof.
\end{proof}
Combining Lemma \ref{7.1}, Lemma \ref{7.2} and Lemma \ref{7.3}, we can express the connection of
the induced graphical metric $\gind$ in terms of the second fundamental form, the singular values and their derivatives. In particular,
the following holds:
\begin{lemma}\label{7.4}
Let $f:U\subset\S^{3} \mapsto \S^{2}$ be a map with non-zero distinct singular values $\lambda$ and $\mu$.
Then,
\begin{eqnarray*}
\nabla^{\gind}_{e_{1}}e_{1}&=&-\frac{h_{11}^{4}}{\lambda}e_{2}-\frac{h_{11}^{5}}{\mu}e_{3}, \\
\nabla^{\gind}_{e_{1}}e_{2}&=&\quad\frac{h_{11}^{4}}{\lambda}e_{1}
-\frac{(1+\lambda^{2})\mu h_{12}^{5}+\lambda(1+\mu^{2})h_{13}^{4}}{\mu^2-\lambda^{2}}e_{3}, \\
\nabla^{\gind}_{e_{1}}e_{3}&=&\quad\frac{h_{11}^{5}}{\mu}e_{1}+\frac{(1+\lambda^{2})\mu h_{12}^{5}+\lambda(1+\mu^{2})h_{13}^{4}}{\mu^{2}-\lambda^{2}}e_{2}.
\end{eqnarray*}
Moreover,
\begin{eqnarray*}
\nabla^{\gind}_{e_{2}}e_{1}&=&-\frac{e_{1}(\arctan \lambda)}{\lambda}e_{2}-\frac{h_{12}^{5}}{\mu}e_{3},\\
\nabla^{\gind}_{e_{2}}e_{2}&=&\quad\frac{e_{1}(\arctan \lambda)}{\lambda}e_{1}
-\frac{\mu(1+\lambda^{2})h_{22}^{5}+\lambda(1+\mu^{2})e_{3}(\arctan \lambda)}{\mu^{2}-\lambda^2}e_{3},\\
\nabla^{\gind}_{e_{2}}e_{3}&=&\quad\frac{h_{12}^{5}}{\mu}e_{1}+\frac{\mu(1+\lambda^{2})h_{22}^{5}+\lambda(1+\mu^{2})e_{3}(\arctan \lambda)}{\mu^{2}-\lambda^{2}}e_{2}.
\end{eqnarray*}
Additionally,
\begin{eqnarray*}
\nabla^{\gind}_{e_{3}}e_{1}&=&-\frac{h_{13}^{4}}{\lambda}e_{2}-\frac{e_{1}(\arctan \mu)}{\mu}e_{3}, \nonumber\\
\nabla^{\gind}_{e_{3}}e_{2}&=&\quad\frac{h_{13}^{4}}{\lambda}e_{1}
-\frac{\mu(1+\lambda^{2})e_{2}(\arctan \mu)+\lambda(1+\mu^{2})h_{33}^{4}}{\mu^2-\lambda^{2}}e_{3}, \nonumber\\
\nabla^{\gind}_{e_{3}}e_{3}&=&\quad\frac{e_{1}(\arctan \mu)}{\mu}e_{1}+\frac{\mu(1+\lambda^{2})e_{2}(\arctan \mu)+\lambda(1+\mu^{2})h_{33}^{4}}{\mu^{2}-\lambda^{2}}e_{2},
\end{eqnarray*}
and
\begin{equation*}
[e_{1}, e_{2}]=\frac{h_{11}^{4}}{\lambda}e_{1}+\frac{e_1(\arctan\lambda)}{\lambda}e_2+\frac{\lambda(1+\mu^{2})}{\mu(\lambda^{2}-\mu^{2})}(\lambda h_{12}^{5}+\mu h_{13}^{4})e_{3},
\end{equation*}
\begin{equation*}
[e_{1}, e_{3}]=\frac{h_{11}^{5}}{\mu}e_{1}+\frac{\mu(1+\lambda^{2})}{\lambda(\mu^{2}-\lambda^{2})}(\lambda h_{12}^{5}+\mu h_{13}^{4})e_{2}+\frac{e_1(\arctan\mu)}{\mu}e_3,
\end{equation*}
where all the frames are with respect to the frames arising from the singular value decomposition.
\end{lemma}
As a consequence of the above formulas, we may express the differential $\varphi$ of the Gauss map
of the submersion $f:(\S^3,\gind)\to\S^2$ in terms of the singular values
and the second fundamental form.
\begin{lemma}\label{7.5}
Let $f:U\subset\S^{3} \mapsto \S^{2}$ be a map with non-zero distinct singular values $\lambda$ and $\mu$.
Then, the tensor $\varphi$ satisfies
\begin{eqnarray*}
\varphi_{12}&=&\gind(\varphi(e_{1}), e_{2})=\lambda^{-1}h_{11}^{4},\\
\varphi_{22}&=&\gind(\varphi(e_{2}), e_{2})=\lambda^{-1}h_{12}^{4},\\
\varphi_{32}&=&\gind(\varphi(e_{3}), e_{2})=\lambda^{-1}h_{13}^{4}.
\end{eqnarray*}
Moreover,
\begin{eqnarray*}
\varphi_{13}&=&\gind(\varphi(e_{1}), e_{3})=\mu^{-1}h_{11}^{5},\\
\varphi_{23}&=&\gind(\varphi(e_{2}), e_{3})=\mu^{-1}h_{12}^{5},\\
\varphi_{33}&=&\gind(\varphi(e_{3}), e_{3})=\mu^{-1}h_{13}^{5},
\end{eqnarray*}
where the frames are arising from the singular value decomposition.
\end{lemma}
\subsection{Bochner-Weitzneb\"ock formulas}
In the sequel, we compute the gradients and the Laplacians of the functions $u_1$ and $u_2$ in the case where
$f$ is a minimal map. Let us point out here that the function $u_1$ is smooth and well defined globally on $M$, even at points where
$f$ is not a submersion. The proofs
are straightforward and for that reason we omit them; for details, we refer to
\cite{wang1}.

\begin{lemma}\label{lapu}
Let $f:U\subset\S^{3} \mapsto \S^{2}$ be a minimal map. The gradient and the Laplacian of the functions $u_1$
are given by
$$
\nabla u_1=-u_1{\sum}_{k=1}^{3}\big(\lambda h_{k2}^{4}+\mu h_{k3}^{5}\big) e_{k}
$$
and
$$
\Delta u_1 =-|A|^{2} u_1 + 2 \lambda\mu\, u_1 {\sum}_{k=1}^3\big(h_{2k}^{4}h_{3k}^{5}-h_{2k}^{5}h_{3k}^{4}\big)-2(\lambda^{2}+\mu^{2})u_1^{3},
$$
where $\{e_1,e_2,e_3\}$ is the frame arising from the singular value decomposition.
\end{lemma}

\begin{lemma}\label{lapv}
Let $f:U\subset\S^{3} \mapsto \S^{2}$ be a minimal submersion. The gradient and the Laplacian of the function $u_2$ are given by:
$$
\nabla u_2 =u_1{\sum}_{k=1}^{3}\big(\mu h_{k2}^{4}+ \lambda h_{k3}^{5}\big) e_{k}
$$
and
\begin{eqnarray*}\label{Eq: Laplacian v 3 sphere}
\Delta u_2\hspace{-8pt}&=&\hspace{-8pt}-u_2{\sum}_{k=1}^3\big(|h_{k2}|^{2}+|h_{k3}|^{2}\big)+2u_1{\sum}_{k=1}^3\big(h_{2k}^{4}h_{3k}^{5}-h_{2k}^{5}h_{3k}^{4}\big)+4u_1^{2}u_2\\
&&+\mu u_1\frac{(h_{11}^{4})^{2}+(h_{12}^{4})^{2}+(h_{13}^{4})^{2}}{\lambda}
+\lambda u_1\frac{(h_{11}^{5})^{2}+(h_{12}^{5})^{2}+(h_{13}^{5})^{2}}{\mu}.
\end{eqnarray*}
\end{lemma}
\begin{lemma}\label{8222021}
Let $f:U\subset\S^{3} \mapsto \S^{2}$ be a minimal submersion.
The gradient and the Laplacian of the angle function $w=u_1+u_2$
satisfy
$$
|\nabla w|^2=(\mu-\lambda)^2u_1^2\big((h_{12}^4-h_{13}^5)^2+(h_{23}^{4}-h_{33}^{5})^{2}+(h_{22}^{4}-h_{23}^{5})^{2}\big)
$$
and
\begin{eqnarray*}
\Delta w&=&-w\big((h_{12}^{4}-h_{13}^{5})^{2}+(h_{23}^{4}-h_{33}^{5})^{2}+(h_{22}^{4}-h_{23}^{5})^{2}\big)\\
&&-w\big((h_{23}^{4}+h_{22}^{5})^{2}+(h_{12}^{5}+h_{13}^{4})^{2}+(h_{33}^{4}+h_{32}^{5})^{2}\big)\\
&&-(\mu-\lambda)u_1\Big(\tfrac{1}{\mu}{\sum}_{k=1}^3(h^5_{1k})^2-\tfrac{1}{\lambda}{\sum}_{k=1}^3(h^4_{1k})^2\Big)\\
&&-2(\mu-\lambda)^{2}u_1^{3},
\end{eqnarray*}
where as usual $\{e_1,e_2,e_3\}$ is the orthonormal frame arising from the singular value decomposition.
\end{lemma}
\section{Proofs of our main theorems}
Here we give the proofs of Theorem \ref{thmc}, Theorem \ref{thma} and Theorem \ref{thmb}. We follow the
notation introduced in the previous sections.
\begin{thmb}\label{cor sigma}
Let $f:\S^3\to\S^2$ be a minimal submersion whose Gauss map $\mathcal{G}$ satisfies the condition
$$
(\mu-\lambda)\big(\mu\big|\operatorname{Im}d\mathcal{G}\big|^2-\lambda\big|\operatorname{Re}d\mathcal{G}\big|^2\big)
\ge 0,
$$
where $0<\lambda\le\mu$ are the non-zero singular values of $df$.
Then, the map $f$ is weakly conformal with totally geodesic fibers. Moreover, there exists an isometry $T:\S^3\to\S^3$ and a 
conformal diffeomorphism $\varPhi:\S^2(1/2)\to\S^2$ such that
$f=\varPhi\circ\varPi\circ T,$
where $\varPi$ is the standard Hopf fibration.
\end{thmb}

\begin{proof}
Let us prove at first that $f$ is weakly conformal. On the contrary, we suppose that
$f$ is not weakly conformal. Let $x_0\in\S^3$ be a point where the globally defined function
$$
w=\frac{1+\lambda\mu}{\sqrt{(1+\lambda^2)(1+\mu^2)}}\le 1
$$
attains its minimum.
Clearly at the point $x_0$ it holds $0<w(x_0)<1$ and so $\lambda(x_0)<\mu(x_0)$. Moreover, in a sufficiently small
neighbourhood of $x_0$ the function $w$ is smooth. Observe now that from Lemma \ref{7.5} and our assumption, we get that
\begin{eqnarray*}
\mathcal{C}&=&(\mu-\lambda)\Big(\tfrac{1}{\mu}{\sum}_{k=1}^3(h^5_{1k})^2-\tfrac{1}{\lambda}{\sum}_{k=1}^3(h^4_{1k})^2\Big)\\
&=&(\mu-\lambda)\big(\mu\big|\operatorname{Im}\varphi\big|^2-\lambda\big|\operatorname{Re}\varphi\big|^2\big)\\
&\ge& 0.
\end{eqnarray*}
Now from the formula for the Laplacian of $w$ in Lemma \ref{8222021}, we obtain
$$
0\le \Delta w(x_0)\le-2(\mu(x_0)-\lambda(x_0))u^3_1(x_0)<0,
$$
which leads to a contradiction. Therefore, $f$ must be weakly conformal and $w=1$. Going back to the Laplacian
of $w$ we get that
$$
h^4_{22}=h^5_{23}=-h^4_{33}\quad\text{and}\quad h_{22}^5=-h^4_{23}=-h_{33}^5.
$$
Because of the minimality, it follows that
$$h^4_{11}=0\quad\text{and} \quad h^5_{11}=0.$$
Hence, $A(e_1,e_1)=0$ and from
Lemma \ref{Aform} we get that
$$\nabla^{\gind}_{e_1}e_1=\nabla^{\S^3}_{e_1}e_1\quad\text{and}\quad B(e_1,e_1)=0.$$
Since $df(e_1)=0$, we have that
$$
0=B(e_1,e_1)=\nabla^f_{e_1}df(e_1)-df(\nabla^{\S^3}_{e_1}e_1)=-df(\nabla^{\S^3}_{e_1}e_1).
$$
Consequently, the integral curves of the vector field $e_1$ are great circles.

According to Heller's Theorem
\cite[Theorem 3.7]{heller}, there exist conformal diffeomorphisms $T:\S^3\to\S^3$ and $\varPhi:\S^{2}(1/2)\to\S^2$
such that $f=\varPhi\circ\varPi\circ T$, where $\varPi$ is the standard Hopf-fibration.
By the classical Liouville's Theorem \cite[Theorem 6.3]{blair}, each conformal diffeomorphism of $\real{3}$
is a composition of similarities and inversions. Since the sphere is conformally equivalent with the extended
euclidean space, it follows that
each conformal transformation $T:\S^3\to\S^3$ can be written as
$T=T_1\circ T_2$ where $T_1\in O(4)$, and after
identifying $\S^3$ with $\real{3}\cup\{\infty\}$ via stereographic projection, $T_2$ takes the
form $T_2 (x)=\varrho x-b$
for some $\varrho>0$ and $b\in\real{4}$; see for example \cite[Proposition 1.1.7]{habermann} or \cite[page 48]{lee}.
Because the fibers of $f$ and $\varPi$ are great circles and $T$ maps the fibers of $f$ onto the fibers of $\varPi$,
we deduce that $\varrho=1$, $b=0$ and $T$ must be an isometry. This completes the proof.
\end{proof}

\begin{thmc}
Let $f_{kl}:\mathbb{S}^3\to\mathbb{S}^2$ be an equivariant minimal submersion. Then $f_{kl}$ is the composition of the
standard Hopf fibration with the dilation from the radius $1/2$ sphere into the unit sphere.
\end{thmc}
\begin{proof}
From Proposition \ref{secsing} the generating function $a:[0,\pi/2]\to[0,\pi]$ satisfies the differential equation
\begin{eqnarray}\label{pppppppp}
0&=&\frac{a_{ss}}{1+a_s^{2}}
+\frac{\cos s\sin s\big({\cos}\, 2s+(l^{2}-k^{2})\sin^{2}a\big)}{\sin^{2}s\cos^{2}s+\sin^{2}a\big(l^{2}\sin^{2}s+k^{2}\cos^{2}s\big)}a_s\nonumber\\
&&\quad\quad\quad\quad\,-\frac{\sin a\cos a\big(k^{2}\cos^{2}s+l^{2}\sin^{2}s\big)}
{\sin^{2}s\cos^{2}s+\sin^{2}a\big(l^{2}\sin^{2}s+k^{2}\cos^{2}s\big)}.
\end{eqnarray}
Since $f_{kl}$ is a submersion, $a(0)=0$ and $a(\pi/2)=\pi$, from Lemma \ref{singinv} we deduce that $a$ must be a
strictly increasing smooth function. In particular, close to $0$, the function $a$ must have
a Taylor expansion of the form
$$a(s)=a_s(0)s+\cdots$$
with $a_s(0)>0$. Observe that close to
$0$, the denominator $D$ of the last two terms in the right-hand side of \eqref{pppppppp} has an asymptotic behavior of the form
$$D(s)=\sin^{2}s\cos^{2}s+\sin^{2}a\big(l^{2}\sin^{2}s+k^{2}\cos^{2}s)\sim O(s^2).$$
By Taylor expanding the nominator $N$ around
the point $0$, we see that
\begin{eqnarray*}
&&\hspace{-20pt}N(s)\hspace{-2pt}=\hspace{-2pt}a_s\cos s\sin s({\cos}\, 2s+(l^{2}-k^{2})\sin^{2}a)\hspace{-2pt}-\hspace{-2pt}\sin a\cos a(k^{2}\cos^{2}s+l^{2}\sin^{2}s)\\
&&\,\,\,\sim O\big(a_s(0)(1-k^2)s\big).
\end{eqnarray*}
Therefore, $k^2=1$. Similarly, we deduce that $l^2=1$. Let us now denote by $\lambda_2$ and $\lambda_3$ the
non-zero singular values of $f_{kl}$. Following the notation and the formulas of Proposition \ref{secsing} for the second fundamental form of an equivariant submersion,
we obtain that $f_{kl}$ has totally geodesic fibers and that
\begin{eqnarray*}
\mathcal{C}&=&(\lambda_3-\lambda_2)\Big(\tfrac{1}{\lambda_3}{\sum}_{k=1}^3(h^5_{1k})^2-
\tfrac{1}{\lambda_2}{\sum}_{k=1}^3(h^4_{1k})^2\Big)
=\frac{(\lambda_3-\lambda_2)^2}{(1+\lambda_2^2)(1+\lambda_3^2)}\\
&\ge& 0.
\end{eqnarray*}
From Theorem \ref{thma}, we deduce that $f_{kl}$ is a weakly conformal minimal submersion with $k^2=l^2=1$.
From Lemma \ref{singinv}, we get that $a(s)=2s$ and so $f_{kl}$ is the standard Hopf-fibration. This completes
the proof.
\end{proof}

\begin{thma}
Let $f:U\to\mathbb{S}^2$ be a minimal map, where $U$ is an open subset of $\S^3$.
If $f$ has constant singular values and $\varGamma(f)$ constant norm of the second fundamental form, then either $f$ is
constant or weakly conformal. If $U=\S^3$ and $f$ non-constant, then there exists an isometry $T:\mathbb{S}^3\to\S^3$ such that $f=\varPhi\circ\varPi\circ T$, where
$\varPi$ is the
Hopf fibration and $\varPhi:\S^2(1/2)\to\S^2$ is the dilation $\varPhi(p)=2p$ from the
sphere $\S^2(1/2)$ of radius $1/2$.
\end{thma}

\begin{proof}

We will prove at first that the singular values $\lambda$ and $\mu$ of $f$ are equal. If $\mu$
is zero then also $\lambda$ is zero and $f$ is a constant map. Suppose now that $\lambda<\mu$
and that $\mu$ is positive. We distinguish the following cases:

{\em Case A:} Suppose at first that $\lambda=0$. Using Lemma \ref{lapu} we get
\begin{equation*}
0=\Delta u_1 =-|A|^{2}u_1-2\mu^{2} u_1^{3}.
\end{equation*}
Hence, $\mu=0$ and $f$ must be constant.

{\em Case B:} Suppose that the singular value $\lambda$ is also positive.
From (\ref{Eq: components h}), we get
$$h_{12}^{4}=h_{13}^{5}=h_{22}^{4}=h_{33}^{5}=h_{23}^{4}=h_{23}^{5}=0.$$
Hence, from minimality, we have that
$$h_{33}^{4}=-h_{11}^{4}\quad\text{and}\quad h_{22}^{5}=-h_{11}^{5}.$$
For the sake of convenience, we set
$$h_{12}^{5}=\chi,\,\, h_{13}^{4}=\psi,\,\, h_{11}^{5}=z,\,\, h_{11}^{4}=\tau.$$
Using Lemma \ref{lapu}, we get
\begin{equation*}
|A|^{2}=-2\lambda\mu\,  \chi \psi-\frac{2(\lambda^{2}+\mu^{2})}{(1+\lambda^{2})(1+\mu^{2})}.
\end{equation*}
Since by assumption, the norm of the second fundamental form is constant, it follows that $ \chi \psi$ is
a negative constant function on $U\subset\mathbb{S}^{3}$.
Applying the Codazzi equation to the triples of vectors $\{e_{1},e_{2},e_{1}\}$, $\{e_{1},e_{3},e_{1}\}$,
$\{e_{1}, e_{2}, e_{2}\}$, $\{e_{1}, e_{3}, e_{3}\}$,
$\{e_{1}, e_{2}, e_{3}\}$, $\{e_{3}, e_{1}, e_{2}\}$,
$\{e_{2},e_{3},e_{3}\}$
and $\{e_{2},e_{3},e_{2}\}$, and making use of Lemma \ref{7.3} and Lemma \ref{7.4},
we obtain
\begin{eqnarray}\label{Eq: Codazzi 10}
&&\hspace{-40pt}e_{2}(\tau)=\tfrac{\lambda(1+\mu^{2})}{\mu^{2}-\lambda^{2}}(\chi^{2}+\psi^{2})
+\tfrac{2\lambda^{2}(1+\mu^{2})}{\mu(\mu^{2}-\lambda^{2})}\chi \psi-\tfrac{1}{\lambda}\tau^{2}
+\tfrac{\lambda(1+\mu^{2})}{\mu^{2}-\lambda^{2}}z^{2}-\tfrac{\lambda}{1+\lambda^{2}},
\label{Eq: Codazzi 10}\\
&&\hspace{-40pt}e_{2}(\tau)=
\tfrac{\lambda(1-\mu^{2})}{(1+\lambda^{2})(1+\mu^{2})}
-\tfrac{2}{\mu}\chi \psi+\tfrac{1}{\lambda}\psi^{2}-\tfrac{\lambda(1+\mu^{2})}{\mu^{2}-\lambda^{2}}\tau^{2}
\label{Eq: Codazzi 8},\\
&&\hspace{-40pt}e_{3}(z)=\tfrac{\mu(1+\lambda^{2})}{\lambda^{2}-\mu^{2}}(\chi^{2}+y^{2})-\tfrac{2\mu^{2}(1+\lambda^{2})}{\lambda(\mu^2-\lambda^{2})}\chi \psi
-\tfrac{1}{\mu}z^{2}
+\tfrac{\mu(1+\lambda^{2})}{\lambda^{2}-\mu^{2}}\tau^{2}-\tfrac{\mu}{1+\mu^{2}}
\label{Eq: Codazzi 11},\\
&&\hspace{-40pt}e_{3}(z)=
\tfrac{\mu(1-\lambda^{2})}{(1+\lambda^{2})(1+\mu^{2})}
-\tfrac{2}{\lambda}\chi \psi+\tfrac{1}{\mu}\chi^{2}+\tfrac{\mu(1+\lambda^{2})}{\mu^2-\lambda^{2}}z^{2}
.\label{Eq: Codazzi 9}
\end{eqnarray}
Moreover,
\begin{eqnarray}
e_{1}(\chi)-e_{2}(z)&=&\frac{2\mu^{2}-\lambda^{2}
+\lambda^{2}\mu^{2}}{\lambda(\mu^{2}-\lambda^{2})} \tau z,
\label{Eq: Codazzi 10 I}\\
e_{1}(z)+e_{2}(\chi)&=&-\frac{2}{\lambda}\tau \chi,
\label{Eq: Codazzi 12}
\end{eqnarray}
and
\begin{eqnarray}
e_{1}(\psi)-e_{3}(\tau)&=&
\frac{2\lambda^{2}-\mu^{2}+\lambda^{2}\mu^{2}}{\mu(\lambda^{2}-\mu^{2})}\tau z
\label{Eq: Codazzi 11 I}\\
e_{1}(\tau)+e_{3}(\psi)&=&-\frac{2}{\mu}\psi z.
\label{Eq: Codazzi 12 I}
\end{eqnarray}
Additionally,
\begin{eqnarray}
e_{2}(\psi)&=&\frac{2\lambda^{2}-\mu^{2}+\lambda^{2}\mu^{2}}{\mu(\lambda^{2}-\mu^{2})}\chi\tau
+
\frac{\mu^{2}(1+\lambda^{2})}{\lambda(\lambda^{2}-\mu^{2})}\tau \psi \label{Eq: Codazzi 13},\\
e_{3}(\chi)&=&\frac{2\mu^{2}-\lambda^{2}+\lambda^{2}\mu^{2}}{\lambda(\mu^{2}-\lambda^{2})}\psi z+
\frac{\lambda^{2}(1+\mu^{2})}{\mu(\mu^{2}-\lambda^{2})}\chi z.
\label{Eq: Codazzi 13 I}
\end{eqnarray}
Substituting (\ref{Eq: Codazzi 10}) in (\ref{Eq: Codazzi 8}) and  (\ref{Eq: Codazzi 11}) in (\ref{Eq: Codazzi 9}), we get the
following equations:
\begin{eqnarray*}
&&\hspace{-20pt}\lambda^{2}(1+\mu^{2})|A^5|^2+(\lambda^{2}\mu^{2}+2\lambda^{2}-\mu^{2})|A^4|^2\hspace{-4pt}=\hspace{-2pt}
\tfrac{4\lambda^{2}(\mu^{2}-\lambda^{2})}{(1+\lambda^{2})(1+\mu^{2})}\hspace{-2pt}-\hspace{-2pt}4\lambda\mu(1+\lambda^{2})\chi\psi,\\
&&\hspace{-20pt}(\lambda^{2}\mu^{2}+2\mu^{2}-\lambda^{2})|A^5|^2+\mu^{2}(1+\lambda^{2})|A^4|^2
\hspace{-4pt}=\hspace{-2pt}\tfrac{4\mu^{2}(\lambda^{2}-\mu^{2})}{(1+\lambda^{2})(1+\mu^{2})}\hspace{-2pt}-\hspace{-2pt}4\lambda\mu(1+\mu^{2})\chi\psi.
\end{eqnarray*}
Let us regard the above as a system with respect to $|A^5|^2$ and $|A^4|^2$. Then,
its main determinant is
$$\mathcal{D}=2(\lambda^{2}-\mu^{2})^{2}.$$
As a consequence, the system has a unique solution which is given by
\begin{eqnarray}\label{Eq: Codazzi 14}
\begin{array}{ll}
\chi^{2}+z^{2}=\frac{1}{\mu^{2}-\lambda^{2}}\big(\tfrac{3\lambda^{2}\mu^{2}+\lambda^{4}\mu^{2}+\lambda^{2}\mu^{4}-\mu^{4}}{(1+\lambda^{2})(1+\mu^{2})}+\lambda\mu(\lambda^{2}\mu^{2}-\mu^{2}-2)\chi \psi\big),\\
\psi^{2}+\tau^{2}=\frac{1}{\lambda^{2}-\mu^{2}}\big(\tfrac{3\lambda^{2}\mu^{2}+\lambda^{2}\mu^{4}+\lambda^{4}\mu^{2}-\lambda^{4}}{(1+\lambda^{2})(1+\mu^{2})}
+\lambda\mu(\lambda^{2}\mu^{2}-\lambda^{2}-2)\chi \psi\big).
\end{array}
\end{eqnarray}
Differentiating the first equation of (\ref{Eq: Codazzi 14}) with respect to $e_{1}$ and $e_{2}$ and using (\ref{Eq: Codazzi 10 I}) and (\ref{Eq: Codazzi 12}), we obtain the following equations:
\begin{eqnarray}\label{Eq: systemdiffalgebraic}
\begin{array}{ll}
\chi e_{1}(\chi)- z e_{2}(\chi)=\frac{2}{\lambda}\tau\chi z,\\
\\
z e_{1}(\chi)+\chi e_{2}(\chi)=\frac{2\mu^{2}-\lambda^{2}+\lambda^{2}\mu^{2}}{\lambda(\mu^{2}-\lambda^{2})}z^{2} \tau.
\end{array}
\end{eqnarray}
Since $\chi \psi < 0$, the main determinant of (\ref{Eq: systemdiffalgebraic}) is
$$\mathcal{D}_{1}= \chi^{2}+z^{2} \neq 0.$$
As a consequence, the system (\ref{Eq: systemdiffalgebraic}) has a unique solution which is given by
\begin{equation}\label{Eq: solutionsystemdiffalgebraic}
e_{1}(\chi)=\big(2 \chi^2+\tfrac{2\mu^{2}-\lambda^{2}+\lambda^{2}\mu^{2}}{\mu^{2}-\lambda^{2}}z^{2}\big)
\tfrac{\tau z}{\lambda(x^{2}+z^{2})}\,\,\,\text{and}\,\,\, e_{2}(\chi)=\tfrac{\lambda(1+\mu^{2}) \chi \tau z^{2}}{(\mu^{2}-\lambda^{2})(x^{2}+z^{2})}.
\end{equation}
As a consequence, we get
\begin{equation}\label{Eq: solutionsystemdiffalgebraicI}
e_{2}(z)=e_{1}(\chi)-\frac{2\mu^{2}-\lambda^{2}+\lambda^{2}\mu^{2}}{\lambda(\mu^{2}-\lambda^{2})} z \tau=\frac{\lambda(1+\mu^{2}) \chi^{2} \tau z}{(\lambda^{2}-\mu^{2})(\chi^{2}+z^{2})}.
\end{equation}
Since the function $\chi \psi$ is constant on $U$, we have $e_{2}(\chi \psi)=0$. Using (\ref{Eq: Codazzi 13}) and (\ref{Eq: solutionsystemdiffalgebraic}), we get
\begin{equation*}
\frac{\tau}{\mu^{2}-\lambda^{2}}\Big(\frac{\lambda(1+\mu^{2})}{x^{2}+z^{2}}\chi \psi z^{2}-\frac{\mu^{2}(1+\lambda^{2})}{\lambda}\chi \psi-\frac{2\lambda^{2}-\mu^{2}+\lambda^{2}\mu^{2}}{\mu}\chi^{2}\Big)=0.
\end{equation*}
If $\tau=0$ on $U$, then (\ref{Eq: Codazzi 8}) gives that the function $\psi$ is constant on
$U$. Using (\ref{Eq: Codazzi 12 I}), we obtain that $z=0$ on $U$. In the sequel, we consider the open subset
$$\mathcal{W}_{1}=\{p \in U: \tau(p) \neq 0\}$$
of $U$. Then,
\begin{equation}\label{Eq: solutionsystemdiffalgebraicII}
\frac{\lambda(1+\mu^{2})}{\chi^{2}+z^{2}}\chi \psi z^{2}-\frac{\mu^{2}(1+\lambda^{2})}{\lambda}
\chi \psi-\frac{2\lambda^{2}-\mu^{2}+\lambda^{2}\mu^{2}}{\mu}\chi^{2}=0.
\end{equation}
on $\mathcal{W}_1$. Differentiating (\ref{Eq: solutionsystemdiffalgebraicII}) with respect to $e_{2}$ and using (\ref{Eq: solutionsystemdiffalgebraic}) and (\ref{Eq: solutionsystemdiffalgebraicI}), we obtain
\begin{equation*}
\tau z^{2}\chi^{2}\Big(\frac{\lambda(1+\mu^{2})}{\chi^{2}+z^{2}}\chi \psi+\frac{2\lambda^{2}-\mu^{2}+\lambda^{2}\mu^{2}}{\mu}\Big)=0.
\end{equation*}
If $z=0$ on $\mathcal{W}_1$, then
$$e_{1}(z)=e_{3}(z)=0$$
on $\mathcal{W}_1$. Using (\ref{Eq: Codazzi 9}), we easily conclude that the function $\chi$ is constant on $\mathcal{W}_1$. Consequently, (\ref{Eq: Codazzi 12}) gives $\tau=0$ on $\mathcal{W}_1$ which is a contradiction. In the sequel, we consider the open subset
$$\mathcal{W}_2=\{p \in \mathcal{W}_1: z(p) \neq 0\}$$
of $\mathcal{W}_1$. Then,
\begin{equation}\label{Eq: solutionsystemdiffalgebraicIII}
\chi\psi = -\frac{2\lambda^{2}-\mu^{2}+\lambda^{2}\mu^{2}}{\lambda \mu (1+\mu^{2})}(\chi^{2}+z^{2})
\end{equation}
on $\mathcal{W}$. Substituting (\ref{Eq: solutionsystemdiffalgebraicIII}) in (\ref{Eq: solutionsystemdiffalgebraicII}), we get
$$-\frac{1}{\mu}+\frac{\mu(1+\lambda^{2})}{\lambda^{2}(1+\mu^{2})}=0,$$
from where we get $\lambda=\mu$, a contradiction. As a consequence, $\tau=z=0$ on $U$.
Then, (\ref{Eq: Codazzi 10}) and (\ref{Eq: Codazzi 11}) give
\begin{eqnarray}
\lambda\mu(\chi^{2}+\psi^{2})+2\lambda^{2}\chi \psi=\frac{\lambda\mu(\mu^{2}-\lambda^{2})}{(1+\lambda^{2})(1+\mu^{2})},\label{Eq: final Codazzi I}\\
\lambda\mu(\chi^{2}+\psi^{2})+2\mu^{2}\chi \psi=\frac{\lambda\mu(\lambda^{2}-\mu^{2})}{(1+\lambda^{2})(1+\mu^{2})}. \label{Eq: final Codazzi II}
\end{eqnarray}
Adding and subtracting (\ref{Eq: final Codazzi I}), (\ref{Eq: final Codazzi II}), we get
$$
\chi^{2}+\psi^{2}=\frac{\mu^{2}+\lambda^{2}}{(1+\lambda^{2})(1+\mu^{2})}\quad\text{and}\quad
\chi\psi=-\frac{\lambda\mu}{(1+\lambda^{2})(1+\mu^{2})}.
$$
Equivalently,
\begin{equation}\label{Eq: final Codazzi III}
\chi+\psi=\pm (\mu-\lambda)u_{1}\quad\text{and}\quad
\chi-\psi=\pm (\mu+\lambda)u_{1}.
\end{equation}
The solution of (\ref{Eq: final Codazzi III}) are the pairs:
$$
\begin{array}{ll}
(\chi, \psi)=(\mu u_{1}, -\lambda u_{1}),& (\chi , \psi)=(-\lambda u_{1}, \mu u_{1}),\\
\\
(\chi, \psi)=(\lambda u_{1}, -\mu u_{1}),& (\chi, \psi)=(-\mu u_{1}, \lambda u_{1}).
\end{array}
$$
We suppose that $\chi=\mu u_{1}$ and $\psi=-\lambda u_{1}$. Then, (\ref{Eq: Codazzi 8}) and (\ref{Eq: Codazzi 9}) give $\lambda^{2}=\mu^{2}=4$ which is a contradiction. If $\chi=-\lambda u_{1}$ and $\psi=\mu u_{1}$, (\ref{Eq: Codazzi 8}) and (\ref{Eq: Codazzi 9}) give
$$3\lambda^{2}+\mu^{2}=\lambda^{2}\mu^{2}\quad\text{and}\quad 3\mu^{2}+\lambda^{2}=\lambda^{2}\mu^{2}.$$ Therefore, $\lambda=\mu$ which is again a contradiction.
The other two cases for the values $\chi$ and $\psi$ are treated similarly.

Let us suppose now that $f$ is non-constant and $U=\S^3$. As in the proof of Theorem \ref{thma}, from the minimality and the gradient of $w\equiv 1$, we deduce that $f$ is a conformal
submersion from the unit euclidean sphere $\S^3$ into the euclidean unit sphere $\S^2$ the with totally geodesic fibers.  From the result of Heller \cite[Theorem 3.7]{heller}, it follows that there exists an isometry
$T$ of the euclidean unit sphere $\S^3$, which maps the kernel of $df$ into the kernel of the differential of the
standard Hopf fibration $\varPi$, and a conformal diffeomorphism $\varPhi:\S^2(1/2)\to\S^2$ such that $f=\varPhi\circ\varPi\circ T$. Denote by $\varrho$ the conformal factor of $\varPhi$. Since $T$ is an isometry, it will map the 
horizontal space of $f$ into the horizontal space of the Hopf fibration. Consequently, if $v$ is a horizontal unit vector
of $f$, then
$$
\lambda^2=\gind_{\S^2}\big(d\varPhi\circ d\varPi\circ T(v),d\varPhi\circ d\varPi\circ T(v)\big)=\varrho^2.
$$
Because $\lambda$ is constant, the conformal factor $\varrho$ must be constant. Since $\varPhi$ is conformal, we get that $\varrho=2$.
This completes the proof.
\end{proof}



\begin{bibdiv}
\begin{biblist}

\bib{almeida}{article}{
   author={de Almeida, S.C.},
   author={Brito, F.G.B.},
   title={Closed $3$-dimensional hypersurfaces with constant mean curvature
   and constant scalar curvature},
   journal={Duke Math. J.},
   volume={61},
   date={1990},
   pages={195--206},
}

\bib{assimos}{article}{
   author={Assimos, R.},
   author={Jost, J.},
   title={The geometry of maximum principles and a Bernstein theorem in codimension 2},
   journal={arXiv:1811.09869},
   date={2018},
   pages={1--27},
}

\bib{baird}{article}{
   author={Baird, P.},
   title={The Gauss map of a submersion},
   conference={
      title={Miniconference on Geometry and Partial Differential Equations},
   },
   book={
      publisher={The Australian National University},
      place={Centre for Mathematical Analysis, Canberra AUS},
   },
   date={1986},
   pages={8--24},
}

\bib{baird1}{book}{
   author={Baird, P.},
   author={Wood, J.},
   title={Harmonic morphisms between Riemannian manifolds},
   series={London Mathematical Society Monographs. New Series},
   volume={29},
   publisher={The Clarendon Press, Oxford University Press, Oxford},
   date={2003},
}

\bib{blair}{book}{
   author={Blair, D.-E.},
   title={Inversion theory and conformal mapping},
   series={Student Mathematical Library
Series Profile},
   publisher={American Mathematical Society (AMS)},
   place={Providence, RI}
   date={2000},
}

\bib{chang}{article}{
   author={Chang, S.},
   title={On minimal hypersurfaces with constant scalar curvatures in $\mathbb{S}^4$},
   journal={J. Differential Geom.},
   volume={37},
   date={1993},
   pages={523--534},
}

\bib{chern1}{book}{
   author={Chern, S.-S.},
   title={Minimal submanifolds in a Riemannian manifold},
   series={University of Kansas, Department of Mathematics Technical Report
   19 (New Series)},
   publisher={Univ. of Kansas Lawrence},
   place={Kan.},
   date={1968},
}

\bib{chern}{article}{
   author={Chern, S.-S.},
   author={do Carmo, M.},
   author={Kobayashi, S.},
   title={Minimal submanifolds of a sphere with second fundamental form of
   constant length},
   conference={
      title={Functional Analysis and Related Fields (Proc. Conf. for M.
      Stone, Univ. Chicago, Chicago, Ill., 1968)},
   },
   book={
      publisher={Springer},
      place={New York},
   },
   date={1970},
   pages={59--75},
}

\bib{ding}{article}{
   author={Ding, Q.},
   author={Jost, J.},
   author={Xin, Y.-L.},
   title={Minimal graphic functions on manifolds of nonnegative Ricci
   curvature},
   journal={Comm. Pure Appl. Math.},
   volume={69},
   date={2016},
   pages={323--371},
}

\bib{eells1}{article}{
   author={Eells, J.},
   title={Minimal graphs},
   journal={Manuscripta Math.},
   volume={28},
   date={1979},
   pages={101--108},
}


\bib{eells}{article}{
   author={Eells, J.},
   author={Sampson, J.H.},
   title={Harmonic mappings of Riemannian manifolds},
   journal={Amer. J. Math.},
   volume={86},
   date={1964},
   pages={109--160},
}

\bib{eells2}{book}{
   author={Eells, J.},
   author={Ratto, A.},
   title={Harmonic maps and minimal immersions with symmetries},
   series={Annals of Mathematics Studies},
   volume={130},
   publisher={Princeton University Press, Princeton, NJ},
   date={1993},
}

\bib{escobales}{article}{
author={Escobales, R.},
title={Riemannian submersions with totally geodesic fibers},
journal={J. Differential Geom.},
volume={10},
date={1975},
pages={253-276},
}

\bib{gage}{article}{
author={Gage, M.},
title={A note on skew-Hopf fibrations},
journal={Proc. Amer. Math. Soc.},
volume={93},
date={1985},
pages={145-150},
}

\bib{gluck1}{article}{
   author={Gluck, H.},
   author={Gu, W.},
   title={Volume-preserving great circle flows on the 3-sphere},
   journal={Geom. Dedicata},
   volume={88},
   date={2001},
   pages={259--282},
}

\bib{gluck2}{article}{
   author={Gluck, H.},
   author={Ziller, W.},
   title={On the volume of a unit vector field on the three-sphere},
   journal={Comment. Math. Helv.},
   volume={61},
   date={1986},
   pages={177--192},
}

\bib{gluck3}{article}{
   author={Gluck, H.},
   author={Warner, F.},
   title={Great circle fibrations of the three-sphere},
   journal={Duke Math. J.},
   volume={50},
   date={1983},
   pages={107-132},
}

\bib{gluck4}{article}{
   author={Gluck, H.},
   author={Warner, F.},
   author={Yang, C.T.}
   title={Division algebras, fibrations of spheres by great spheres and the topological
   determination of space by the gross behavior of its geodesics},
   journal={Duke Math. J.},
   volume={50},
   date={1983},
   pages={1041-1076},
}

\bib{gluck5}{article}{
   author={Gluck, H.},
   author={Warner, F.},
   author={Ziller, W.}
   title={The geometry of the Hopf fibrations},
   journal={Enseign. Math. (2)},
   volume={32},
   date={1986},
   pages={173-198},
}

\bib{gluck6}{article}{
   author={Gluck, H.},
   author={Warner, F.},
   author={Ziller, W.}
   title={Fibrations of spheres by parallel great spheres and Berger's rigidity theorem},
   journal={Ann. Global Anal. Geom.},
   volume={5},
   date={1987},
   pages={53-82},
}

\bib{habermann}{book}{
author={Habermann, L.},
title={Riemannian metrics of constant mass and moduli spaces of conformal structures},
journal={Lect. Notes Math.},
volume={1743}
publisher={Berlin: Springer}
date={2000},
}	

\bib{h-s1}{article}{
   author={Hasanis, Th.},
   author={Savas-Halilaj, A.},
   author={Vlachos, Th.},
   title={On the Jacobian of minimal graphs in $\Bbb R^4$},
   journal={Bull. Lond. Math. Soc.},
   volume={43},
   date={2011},
   pages={321--327},
}
		

\bib{h-s2}{article}{
   author={Hasanis, Th.},
   author={Savas-Halilaj, A.},
   author={Vlachos, Th.},
   title={Minimal graphs in $\Bbb R^4$ with bounded Jacobians},
   journal={Proc. Amer. Math. Soc.},
   volume={137},
   date={2009},
   pages={3463--3471},
}

\bib{heller}{article}{
author={Heller, S.},
title={Conformal fibrations of $S^3$ by circles},
journal={In: Loubeau, E., Montaldo, S. (eds.) Harmonic
Maps and Differential Geometry, Contemp. Math. 542, AMS, Providence},
date={2011},
pages={195-202}
}

\bib{hopf}{article}{
author={Hopf, H.},
title={\"Uber die Abbildungen der dreidimensionalen Sph\"are auf die Kugelfl\"ache},
journal={Math. Ann},
volume={104},
date={1931},
pages={637-665}
}

\bib{jost5}{article}{
   author={Jost, J.},
   author={Xin, Y.-L.},
   author={Yang, L.},
   title={Submanifolds with constant Jordan angles and rigidity of the
   Lawson-Osserman cone},
   journal={Asian J. Math.},
   volume={22},
   date={2018},
   pages={75--109},
}

\bib{jost1}{article}{
   author={Jost, J.},
   author={Xin, Y.-L.},
   author={Yang, L.},
   title={A spherical Bernstein theorem for minimal submanifolds of higher
   codimension},
   journal={Calc. Var. Partial Differential Equations},
   volume={57},
   date={2018},
   number={6},
   pages={Art. 166, 21},
}
		

\bib{jost2}{article}{
   author={Jost, J.},
   author={Xin, Y.-L.},
   author={Yang, L.},
   title={The geometry of Grassmannian manifolds and Bernstein-type theorems
   for higher codimension},
   journal={Ann. Sc. Norm. Super. Pisa Cl. Sci. (5)},
   volume={16},
   date={2016},
   pages={1--39},
}
		

\bib{jost3}{article}{
   author={Jost, J.},
   author={Xin, Y.-L.},
   author={Yang, L.},
   title={Curvature estimates for minimal submanifolds of higher codimension
   and small G-rank},
   journal={Trans. Amer. Math. Soc.},
   volume={367},
   date={2015},
   pages={8301--8323},
}
		

\bib{jost4}{article}{
   author={Jost, J.},
   author={Xin, Y.-L.},
   author={Yang, L.},
   title={The Gauss image of entire graphs of higher codimension and
   Bernstein type theorems},
   journal={Calc. Var. Partial Differential Equations},
   volume={47},
   date={2013},
   pages={711--737},
}

\bib{lee}{article}{
   author={Lee, J.},
   author={Parker, T.H.},
   title={The Yamabe problem},
   journal={Bull. Am. Math. Soc.},
   volume={17},
   date={1987},
   pages={37-91},
}

\bib{mckay}{article}{
author={McKay, B.}
title={The Blaschke conjecture and great circle fibrations of spheres},
journal={Amer. J. Math.},
volume={126},
date={2004}
pages={1155-1191}
}

\bib{peng1}{article}{
   author={Peng, C.K.},
   author={Terng, C.L.},
   title={Minimal hypersurfaces of spheres with constant scalar curvature},
   conference={
      title={Seminar on minimal submanifolds},
   },
   book={
      series={Ann. of Math. Stud.},
      volume={103},
      publisher={Princeton Univ. Press},
      place={Princeton, NJ},
   },
   date={1983},
   pages={177--198},
}


\bib{peng2}{article}{
   author={Peng, C.K.},
   author={Terng, C.L.},
   title={The scalar curvature of minimal hypersurfaces in spheres},
   journal={Math. Ann.},
   volume={266},
   date={1983},
   pages={105--113},
}


\bib{savas4}{article}{
   author={Savas-Halilaj, A.},
   author={Smoczyk, K.},
   title={Mean curvature flow of area decreasing maps between Riemann
   surfaces},
   journal={Ann. Global Anal. Geom.},
   volume={53},
   date={2018},
   pages={11--37},
}

\bib{savas3}{article}{
   author={Savas-Halilaj, A.},
   author={Smoczyk, K.},
   title={Evolution of contractions by mean curvature flow},
   journal={Math. Ann.},
   volume={361},
   date={2015},
   pages={725--740},
}
\bib{savas2}{article}{
   author={Savas-Halilaj, A.},
   author={Smoczyk, K.},
   title={Homotopy of area decreasing maps by mean curvature flow},
   journal={Adv. Math.},
   volume={255},
   date={2014},
   pages={455--473},
}

\bib{savas1}{article}{
   author={Savas-Halilaj, A.},
   author={Smoczyk, K.},
   title={Bernstein theorems for length and area decreasing minimal maps},
   journal={Calc. Var. Partial Differential Equations},
   volume={50},
   date={2014},
   pages={549--577},
}

\bib{schoen2}{article}{
   author={Schoen, R.},
   title={The role of harmonic mappings in rigidity and deformation
   problems},
   conference={
      title={Complex geometry},
      address={Osaka},
      date={1990},
   },
   book={
      series={Lecture Notes in Pure and Appl. Math.},
      volume={143},
      publisher={Dekker},
      place={New York},
   },
   date={1993},
   pages={179--200},
}

\bib{swx}{article}{
   author={Smoczyk, K.},
   author={Wang, G.},
   author={Xin, Y.-L.},
   title={Bernstein type theorems with flat normal bundle},
   journal={Calc. Var. Partial Differential Equations},
   volume={26},
   date={2006},
   pages={57--67},
}

\bib{ucci}{article}{
author={Ucci, J.}
title={On the nonexistence of Riemannian submersions from $CP(7)$ and $QP(3)$},
journal={Proc. Amer. Math. Soc.},
volume={88},
date={1983}
pages={698-700}
}

\bib{yang1}{article}{
author={Yang, C.T.}
title={Smooth great circle fibrations and an application to the topological Blaschke conjecture},
journal={Trans. Amer. Math. Soc.},
volume={320},
date={1990}
pages={504-524}
}

\bib{yang2}{article}{
author={Yang, C.T.}
title={On smooth great circle fibrations of a round sphere},
journal={Differential geometry (Shanghai, 1991) World Sci. Publishing, River Edge, NJ},
date={1993}
pages={301-309}
}


\bib{xin}{book}{
   author={Xin, Y.-L.},
   title={Geometry of harmonic maps},
   series={Progress in Nonlinear Differential Equations and their
   Applications, {\bf 23}},
   publisher={Birkh\"auser Boston Inc.},
   place={Boston, MA},
   date={1996},
}

\bib{wangg}{article}{
   author={Wang, G.},
   title={$\S^1$-invariant harmonic maps from $\S^3$ to $\S^2$},
   journal={Bull. London Math. Soc.},
   volume={32},
   date={2000},
   pages={729--735},
}

\bib{wangm}{article}{
   author={Wang, M.-T.},
   title={On graphic Bernstein type results in higher codimension},
   journal={Trans. Amer. Math. Soc.},
   volume={355},
   date={2003},
   pages={265--271},
}

\bib{wang}{article}{
   author={Wang, M.-T.},
   title={Long-time existence and convergence of graphic mean curvature flow
   in arbitrary codimension},
   journal={Invent. Math.},
   volume={148},
   date={2002},
   pages={525--543},
}

\bib{wang1}{article}{
   author={Wang, M.-T.},
   title={Mean curvature flow of surfaces in Einstein four-manifolds},
   journal={J. Differential Geom.},
   volume={57},
   date={2001},
   pages={301--338},
}

\bib{zawadzki}{article}{
   author={Zawadzki, T.},
   title={On conformal submersions with geodesic or minimal fibers},
   journal={Ann. Glob. Anal. Geom.},
   volume={58},
   date={2020},
   pages={191-205},
}

\bib{zawadzki2}{article}{
   author={Zawadzki, T.},
   title={Existence conditions for conformal submersions with totally umbilical fibers},
   journal={Differ. Geom. Appl.},
   volume={35},
   date={2014},
   pages={69-85},
}
		
\end{biblist}
\end{bibdiv}
\end{document}